\documentclass[10pt]{amsart}
\usepackage{amsmath,amssymb,latexsym,cancel,rotating}
\usepackage{mathrsfs}
\usepackage{enumitem}
\usepackage{needspace}
\usepackage[hidelinks]{hyperref}
\usepackage[expansion=false]{microtype}
\usepackage{graphicx,mathrsfs,color,fancyhdr,amsthm}
\usepackage{tikz}
\usepackage{tikz-cd}
\usepackage{geometry}
\usepackage{hyperref}
\usepackage{cite}
\usepackage{setspace}
\newtheorem{theorem}{Theorem}[section]
\newtheorem{lemma}[theorem]{Lemma}

\newtheorem{definition}[theorem]{Definition}
\newtheorem{proposition}[theorem]{Proposition}

\newtheorem{remark}[theorem]{Remark}

 \def\cD{{\mathcal D}}    
\def\cH{{\mathcal H}}

\def\cK{{\mathcal K}}
\def\bbN{{\mathbb N}}  \def\bbZ{{\mathbb Z}}  \def\bbQ{{\mathbb Q}}
    \def\bbF{{\mathbb F}}
\def\bbC{{\mathbb C}}

        \def\bfU{{\bf U}}
   
  \def\leq{\leqslant}  \def\geq{\geqslant}

\def\Hom{\mbox{\rm Hom}}  
       \def\im{\mbox{\rm Im}\,}
\def\Ext{\mbox{\rm Ext}}   
\def\dim{\mbox{\rm dim}}   \def\End{\mbox{\rm End}}

\def\Aut{\mbox{\rm Aut}}

  \def\tr{\mbox{\rm tr}}

\def\bfV{{\mathbf V}}

\def\bfH{{\mathbf H}}

\def\bfV{\mathbf{V}}
\def\bfW{\mathbf{W}}
\def\bfE{\mathbf{E}}
\def\bfG{\mathbf{G}}

\def\bfU{\mathbf{U}}

\def\Ind{\mathbf{Ind}}

\DeclareMathOperator{\GL}{GL}
\DeclareMathOperator{\rep}{rep}
\DeclareMathOperator{\ind}{ind}
\DeclareMathOperator{\nilrep}{nilrep}

\newcommand{\N}{\mathbb N}
\newcommand{\Z}{\mathbb Z}
\newcommand{\ddim}{\mathop{\underline{\mathrm{dim}}}\nolimits}

\newcommand{\cT}{\mathcal T}
\newcommand{\cG}{\mathcal G}
\newcommand{\bfc}{\mathbf c}
\newcommand{\bfd}{\mathbf d}
\newcommand\bbq{{\mathbb Q}}

\newcommand\mf{{\mathcal{F}}}

\newcommand\mo{{\mathcal{O}}}
\newcommand\mP{{\mathcal{P}}}

\newcommand\ml{{\mathcal{L}}}

\begin{document}

\title[On Lusztig sheaves for tame quivers]{On Lusztig Sheaves for Tame Quivers}

\author{Yumeng Wu, Jie Xiao}
\address{Beijing International Center for Mathematical Research, Beijing 100871, P. R. China}
\email{2506397175@pku.edu.cn (Y. Wu)}
\address{School of Mathematical Sciences, Beijing Normal University, Beijing 100875, P. R. China}
\email{jxiao@bnu.edu.cn (J. Xiao)}





\bibliographystyle{abbrv}
\begin{abstract}
Let $Q$ be a connected quiver of tame type, and let $U_v^-$ be the positive
part of the corresponding quantized affine algebra. In Lusztig's construction
of the canonical basis of $U_v^-$, the simple perverse sheaves occurring,
up to shift, as direct summands of the flag sheaf complexes are given by
intersection cohomology complexes
$\operatorname{IC}(X(c,m),\mathcal{L}_{\chi})$.

We study the restrictions of these intersection cohomology complexes to the
locally closed subsets $Y(c,m)$ obtained by fixing the preprojective,
non-homogeneous regular, and preinjective summands in the category
$\operatorname{rep} Q$. We prove pointwise purity on the loci where the
homogeneous regular part is a direct sum of pairwise non-isomorphic
quasi-simple modules, and establish a semisimple decomposition of the
restrictions to the corresponding larger strata.

The graded multiplicities of the local systems associated with irreducible
characters of symmetric groups determine the transition coefficients from
the canonical basis to the PBW basis constructed via quiver representations.
This gives a geometric explanation of the positivity of these coefficients
and the unitriangularity of the transition matrix.
\end{abstract}
\maketitle
\section*{Introduction}

Ringel~\cite{Ringel1990x} realized the negative part $U_v^-$ of a quantum group using Hall algebras of Dynkin quivers, and Green~\cite{Green1995} extended this realization to arbitrary finite acyclic quivers by identifying $U_v^-$ with the corresponding generic twisted composition algebra. For quivers of Dynkin type, the orbits in the representation varieties correspond to elements of a Poincar\'e--Birkhoff--Witt (PBW) basis. Lusztig~\cite{lusztig1990canonical,Lusztig1991} constructed the canonical basis by algebraic and geometric methods. In the geometric construction, its elements are represented by simple perverse sheaves on the representation varieties.

For a connected quiver $Q$ of tame type, Lusztig~\cite{Lusztig1992} used the representation theory of $Q$ to describe the simple perverse sheaves that realize the canonical basis of the corresponding quantum affine algebra. These sheaves have the form $\operatorname{IC}(X(c,m),\mathcal L_\chi)$, with support $\overline{X(c,m)}$, where the local systems $\mathcal L_\chi$ are associated with irreducible characters of symmetric groups. Further details, including the description for arbitrary orientations, can be found in~\cite{li2007notesaffinecanonicalmonomial}.

For quantum affine algebras, canonical bases can also be described through algebraic constructions of PBW bases; see~\cite{BCP1999,BN2004,MT2018,SZ2026}. Constructions using representations of tame quivers were developed in~\cite{linxiao,XIAO2023510}, with the cyclic-quiver results in~\cite{DengDuXiao2007} as an ingredient. In particular, Xiao, Xu, and Zhao~\cite{XIAO2023510} proved that the bar-invariant basis obtained from their PBW basis coincides with Lusztig's canonical basis and identified the corresponding indices.

Lusztig's work\cite{lusztig1990canonical} and Kato's work~\cite{Kato2014} show how a categorical description of PBW bases can explain the positivity and unitriangularity of the transition matrix from the canonical basis to a PBW basis. We seek such a description using constructible complexes on quiver representation varieties and their relation to Lusztig sheaves. In this paper, we establish this relation for connected quivers of tame type by studying restrictions of Lusztig sheaves and their Frobenius weights.

We describe the results more precisely. Let $\bfE_\nu$ be the representation variety of $Q$ over $\overline{\mathbb F}_q$, and let $\delta$ be the minimal positive imaginary root. Write $c=c_P\oplus c_{\mathrm{nh}}\oplus c_I$ for a representation with no homogeneous regular summand. For $\dim c+m\delta=\nu$, let $Y(c,m)$ consist of the representations $c\oplus R$ with $R$ homogeneous regular of dimension $m\delta$. These subsets form a finite stratification of $\bfE_\nu$. Inside $Y(c,m)$, let $X(c,m)$ be the open dense locus where $R$ is a sum of $m$ pairwise non-isomorphic homogeneous regular simples. Ordering these simples defines a symmetric-group covering and local systems $\mathcal L_\chi$, indexed by $\chi\in X(S_m)$ of the symmetric group $S_m,$ where $X(S_m)$ denotes the set of isomorphism classes of irreducible
$\overline{\mathbb{Q}}_\ell$-representations of $S_m.$

Our first result controls the constituents of restricted flag sheaves. The restriction of any flag sheaf complex to $Y(c,m)$ belongs to the triangulated category generated by the restrictions of $\operatorname{IC}(X(c,m),\mathcal L_\chi)$, together with constant rank-one Weil factors (Theorem~\ref{local}). Our main purity theorem states that every normalized flag sheaf complex is pointwise pure of weight zero on $X(c,m)$ (Theorem~\ref{pure}). Combined with the description of constituents, this implies that its restriction to $Y(c,m)$ is pure and geometrically semisimple, with summands given by shifts and twists of these restricted intersection cohomology complexes (Proposition~\ref{prop4.2}).

These results determine the transition coefficients geometrically. For a canonical-basis complex $\operatorname{IC}_\beta$ and $j^{c,m}\colon Y(c,m)\hookrightarrow E_\nu$, write its geometric restriction as
\[
 (j^{c,m})^*\operatorname{IC}_\beta
 \simeq
 \bigoplus_{\chi,r}
 \left((j^{c,m})^*\operatorname{IC}(X(c,m),\mathcal L_\chi)
 [r](r/2)\right)^{\oplus b_{c,\chi,r}^{\beta}}.
\]
The nonnegative integers $b_{c,\chi,r}^{\beta}$ are equivalently graded local-system multiplicities on $X(c,m)$. Theorem~\ref{the4.7} identifies
\[
 f_{c,\chi}^{\beta}(v)
   =\sum_r b_{c,\chi,r}^{\beta}v^{-r}
   \in\mathbb N[v,v^{-1}]
\]
with the coefficients in the expansion of the canonical trace function in the stratum-supported family $N(c,t_\chi)$, under the convention $v=-\sqrt q$. Using the relation between this family and the PBW basis $E(c,t_\chi)$ from~\cite[Lemma~9.13]{XIAO2023510}, the coefficients with aperiodic indices are precisely the canonical-to-PBW transition coefficients. Thus the restriction formula gives a geometric explanation of positivity and unitriangularity, with respect to the support order, for the representation-theoretic PBW basis.

In his lecture notes, Schiffmann~\cite[\S3.7]{OSNotes} suggests that Lusztig's result on stalkwise Frobenius eigenvalues in finite type should extend to tame quivers. Our Theorem~\ref{pure} gives a partial answer concerning their absolute values: for a normalized flag sheaf complex, every Frobenius eigenvalue $\lambda$ on its degree-$i$ stalk cohomology at an $\mathbb F_q$-rational point of $X(c,m)$ satisfies $|\lambda|=q^{i/2}$. Thus we establish pointwise purity on these loci; we do not assert that the eigenvalues themselves are all equal to $q^{i/2}$.

We note that, independently of the present work Xu and
Zhao~\cite{XXZ2026Frobenius} have established pointwise purity
for Lusztig sheaves of tame quivers using a different approach.

Analogous positivity and unitriangularity results for the PBW bases
constructed using Khovanov--Lauda--Rouquier (KLR) algebras were obtained
in finite ADE type~\cite{Kato2014,10.1215/00127094-2681278}.
In affine type, Kleshchev and McNamara developed the theory of (dual) PBW bases
via KLR algebras~\cite{Kleshchev2014,McNamara2017}.
Within this framework, McNamara further established positivity and
unitriangularity for the expansion of canonical basis elements
in terms of PBW bases in symmetric affine type~\cite{McNamara2017}.

The main remaining step is to prove purity in the presence of non-homogeneous regular tubes. The proof uses contracting transverse slices and induction over a finite stratification in regular tubes. The required tube filtrations mix prescribed semisimple regular quotients with unrestricted quotients of dimension $\delta$; their pointwise purity is established in Lemma~\ref{lem:tube-filtration-affine-paving}. The restriction and slice arguments then give the purity theorem for general tame quivers, including the periodic tube representations occurring in the intermediate strata (Theorem~\ref{pure}). Related results on purity and affine pavings of flag varieties can be found in~\cite{DCLP1988,Sauter2017,ES2022}.

Our results extend the geometric comparison for the Kronecker quiver in~\cite{McGerty2005,WX2026} to all connected quivers of tame type.

In Section~1, we recall the representation theory and the strata $X(c,m)$ and $Y(c,m)$. In Section~2, we review Lusztig sheaves and the relevant bases. In Section~3, we prove the restriction theorem. In Section~4, we derive the decomposition and transition formulas from pointwise purity, whose proof occupies Section~5.

Throughout this paper, we follow the notation and conventions of \cite{SS,lusztig2010introduction,Pramod-2021} for quiver representation theory, quantum groups, and algebraic geometry.
\section{Preliminaries on representations of tame quivers}\label{sec1}
This section reviews basic properties of representations of tame quivers and the stratification of their representation varieties. Throughout, let $k=\overline{\bbF_q}$ be an algebraic closure of a finite field.
\subsection{Conventions and the Euler form}
All representations are finite-dimensional over $k$.
We identify $\rep_k Q$ with the category of finite-dimensional left
$kQ$-modules, and write $\ind\mathcal C$ for the isomorphism classes of
indecomposable objects of a category $\mathcal C$.

\begin{definition}
In this paper, a \emph{tame quiver} is a finite connected acyclic quiver
whose underlying unoriented graph is an extended Dynkin diagram
\[
 \widetilde A_n\ (n\geq 1),\qquad
 \widetilde D_n\ (n\geq 4),\qquad
 \widetilde E_6,\quad \widetilde E_7,\quad \widetilde E_8.
\]
Here $\widetilde A_1$ denotes the Kronecker quiver, with two vertices
and two parallel arrows. Thus \emph{tame} is used in the
representation-infinite sense; Dynkin quivers are excluded.
This is the Euclidean case in the classification of hereditary algebras;
see \cite[Chapter~XIII]{SS}.
\end{definition}

Write $Q=(Q_0,H,s,t)$, where $Q_0$ is the vertex set, $H$ is the arrow set,
and $s,t:H\to Q_0$ are the source and target maps.
For a representation $M=((M_i)_{i\in Q_0},(M_h)_{h\in H})$, its
\emph{dimension vector} is
\[
 \ddim M=(\dim_k M_i)_{i\in Q_0}\in\N^{Q_0}.
\]
For an $Q_0$-graded vector space $V=\bigoplus_{i\in Q_0}V_i$, we similarly
write $\ddim V=(\dim_k V_i)_{i\in Q_0}$.
For $\alpha,\beta\in\Z^{Q_0}$, set
\[
 \langle\alpha,\beta\rangle
 =\sum_{i\in Q_0}\alpha_i\beta_i
  -\sum_{h\in H}\alpha_{s(h)}\beta_{t(h)}.
\]
Since $kQ$ is hereditary, $\Ext^j_{kQ}(M,N)=0$ for
$M,N\in\rep_k Q$ and $j\geq2$. Moreover,
\begin{equation}
 \langle\ddim M,\ddim N\rangle
 =\dim_k\Hom_{kQ}(M,N)-\dim_k\Ext^1_{kQ}(M,N).
 \label{eq:euler}
\end{equation}
See \cite[\S1.2]{linxiao}.
Let $\delta$ be the primitive positive imaginary root.
We use the defect convention
$\partial(M)=\langle\delta,\ddim M\rangle$.
Unless indicated otherwise, $\Hom$, $\Ext^1$, and $\End$ are taken over $kQ$.

\subsection{Representation spaces and base change}

Fix an $Q_0$-graded vector space $\bfV=\bigoplus_{i\in Q_0}V_i$ with
$\ddim V=\nu$. Define the affine representation space and the base-change
group by
\begin{equation}
 \bfE_\nu=\bfE_{\bfV}=\bigoplus_{h\in H}\Hom_k(V_{s(h)},V_{t(h)}),
 \qquad
 \bfG_\nu=\prod_{i\in Q_0}\mathbf{G}l(V_i).
 \label{eq:representation-space}
\end{equation}
For $g=(g_i)_{i\in Q_0}\in \bfG_\nu$ and $x=(x_h)_{h\in H}\in \bfE_\nu$, the action is
\begin{equation}
 (g\cdot x)_h=g_{t(h)}x_hg_{s(h)}^{-1}
 \qquad(h\in H).
 \label{eq:base-change}
\end{equation}
The point $x$ determines the representation
$M_x=((V_i)_{i\in Q_0},(x_h)_{h\in H})$.
Two points lie in the same $\bfG_V$-orbit if and only if their representations
are isomorphic. The stabilizer of $x$ is $\Aut_{kQ}(M_x)$, and hence
\begin{equation}
 \dim(\bfG_V\cdot x)=\dim \bfG_V-\dim_k\End(M_x),
 \qquad
 \dim \bfG_V-\dim \bfE_\nu=\langle\nu,\nu\rangle.
 \label{eq:orbit-dimension}
\end{equation}
See \cite[\S2.1]{linxiao} and \cite[\S7.3]{XIAO2023510}.
All closures and density statements below refer to the Zariski topology.

\subsection{Indecomposables and tubes}

Let $\tau$ be the Auslander--Reiten translation.
An indecomposable representation $M$ is \emph{preprojective} if
$M\simeq\tau^{-r}P(i)$, and \emph{preinjective} if
$M\simeq\tau^r I(i)$, for some $r\geq0$ and $i\in Q_0$;
here $P(i)$ and $I(i)$ are the indecomposable projective and injective modules
associated with $i$. The remaining indecomposables are \emph{regular}.
Let $\mathbf{P},\mathbf{R},\mathbf{I}$ denote their additive subcategories.
For an indecomposable representation $M$, this classification can also be expressed as
\begin{equation}
 M\in\mathbf{P}\ \Longleftrightarrow\ \partial(M)<0,\qquad
 M\in\mathbf{R}\ \Longleftrightarrow\ \partial(M)=0,\qquad
 M\in\mathbf{I}\ \Longleftrightarrow\ \partial(M)>0.
 \label{eq:defect}
\end{equation}
These assertions are recalled in \cite[\S5.1, Lemma~5.1]{XIAO2023510}.

The regular category decomposes into tubes
$\cT_z$, indexed by $z\in\mathbb P^1(k)$ after a choice of parametrization.
Each $\cT_z$ is an extension-closed abelian subcategory and admits an exact
equivalence
\[
 \cT_z\simeq\nilrep_k C_{r_z},
\]
where $C_r$ is the cyclic quiver with vertices $\Z/r\Z$ and arrows
$i\longrightarrow i+1$. The integer $r_z$ is the \emph{rank} of the tube.
The tubes of rank one are \emph{homogeneous}; the others are
\emph{non-homogeneous}. There are finitely many non-homogeneous tubes, at parameters
$z_1,\ldots,z_e$; put $\cT_j=\cT_{z_j}$ and $r_j=r_{z_j}>1$.
Thus $e$ denotes the number of non-homogeneous tubes.
We write
\[
 Z_{  h}=\{z\in\mathbb P^1(k):r_z=1\},\qquad
 \mathbf{R}_{nh}=\bigoplus_{j=1}^{e}\cT_j,\qquad
 \mathbf{R}_{  h}=\bigoplus_{z\in Z_{  h}}\cT_z.
\]
Only finite direct sums of objects are meant here.
The ranks are as follows; entries equal to one are omitted:
\[
\begin{array}{c|c}
 \text{type}&\text{non-homogeneous tube ranks}\\ \hline
 \widetilde A_{a+b-1}&a,b\\
 \widetilde D_n&2,2,n-2\\
 \widetilde E_6&2,3,3\\
 \widetilde E_7&2,3,4\\
 \widetilde E_8&2,3,5
\end{array}
\]
In the first row, $a,b>0$ count the arrows in the two directions around
the underlying cycle; the Kronecker case has $a=b=1$.
For the tube classification and ranks, see \cite{DlabRingel1976} and \cite[Chapter~XIII]{SS};
for the cyclic-quiver equivalences, see \cite[\S5.3]{XIAO2023510}.

Fix such an equivalence for each tube.
Denote the images of the vertex simples by $S_{z,i}$.
Let $S_{z,i}[l]$ be the indecomposable of regular length $l\geq1$
and regular top $S_{z,i}$, with $i\in\Z/r_z\Z$.
These are all the indecomposable objects of $\cT_z$.
For $z\in Z_{  h}$ write $R_z[l]=S_{z,0}[l]$; then
$\ddim R_z[l]=l\delta$.
Consequently, every representation has a decomposition, unique up to
isomorphism of its summands,
\begin{equation}
 M\simeq M_{P}\oplus M_{nh}\oplus M_{hr}\oplus M_{I}.
 \label{eq:four-parts}
\end{equation}
Here, $M_{P}$, $M_{nh}$, $M_{hr}$, and $M_{I}$ denote the preprojective, nonhomogeneous regular, homogeneous regular, and preinjective direct summands of $M$, respectively.
\subsection{Hom and Ext}

For $P\in\mathbf{P}$, $R\in\mathbf{R}$, and $I\in\mathbf{I}$,
\begin{align}
 \Hom(R,P)&=\Hom(I,P)=\Hom(I,R)=0, \label{eq:hom-separation}\\
 \Ext^1(P,R)&=\Ext^1(P,I)=\Ext^1(R,I)=0. \label{eq:ext-separation}
\end{align}
See the separation theory in \cite[Chapters~XI--XIII]{SS}.
If $X\in\cT_z$ and $Y\in\cT_w$ with $z\ne w$, then
\begin{equation}
 \Hom(X,Y)=\Hom(Y,X)=\Ext^1(X,Y)=\Ext^1(Y,X)=0.
 \label{eq:tube-orthogonality}
\end{equation}
Indeed, the tubes have disjoint sets of regular composition factors,
and $\tau$ preserves each tube. Apply the hereditary
Auslander--Reiten formula
\begin{equation}
 D\Ext^1(M,N)\simeq\Hom(N,\tau M)
                 \simeq\Hom(\tau^{-1}N,M),
 \qquad D=\Hom_k(-,k),
 \label{eq:ar}
\end{equation}
with $\tau P=0$ for projective $P$ and $\tau^{-1}I=0$ for injective $I$;
see \cite[Chapter~IV]{ASS}.
Every preprojective or preinjective indecomposable $X$ satisfies
$\End(X)\simeq k$ and $\Ext^1(X,X)=0$.

\subsection{An order compatible with the vanishing properties}

Put $n=|Q_0|$ and number the vertices $1,\ldots,n$ so that every arrow
goes from a vertex with a smaller label to one with a larger label. We use the periodic reflection-functor enumerations
\begin{equation}
 P_{rn+j-1}=\tau^{-r}P(n+1-j),\qquad
 I_{rn+j}=\tau^r I(j)
 \quad(r\geq0,\ 1\leq j\leq n).
 \label{eq:enumerations}
\end{equation}
Thus the standard orders on the two components are
\[
 P_0<P_1<P_2<\cdots,\qquad
 \cdots<I_3<I_2<I_1.
\]
They are the orders associated with the periodic admissible sequence
in \cite[\S5.2]{XIAO2023510}. Their Hom and Ext properties are also recorded in
\cite[\S4.3, p.~844]{linxiao}.

\begin{definition}
Choose a total order $\triangleleft$ on the tube index set.
Define a strict partial order $<$ on $\ind(\rep_k Q)$ by
\begin{enumerate}[label=\textup{(\roman*)},leftmargin=*]
 \item $P_a<P_b$ if $a<b$, and $I_b<I_a$ if $a<b$;
 \item $P<R<I$ for every indecomposable
       $P\in\mathbf{P}$, $R\in\mathbf{R}$, and $I\in\mathbf{I}$, and $P<I$;
 \item $X<Y$ for $X\in\ind\cT_z$ and $Y\in\ind\cT_w$
       whenever $z\triangleleft w$.
\end{enumerate}
Distinct indecomposables in a single tube are left incomparable.
\end{definition}
For any dimension vector $\nu,$ there exist an indecomposable preprojective module $P$ and an indecomposable preinjective module $I,$ with the following property. For every $P'>P,$ where $P'$ is indecomposable preprojective, and every $I'<I,$ where $I'$ is indecomposable preinjective, we have $\ddim P'>\nu$ and $\ddim I'>\nu.$
\begin{proposition}\label{homext}
For indecomposables $M,N$, this order satisfies
\begin{equation}
 M<N\quad\Longrightarrow\quad
 \Hom(N,M)=0=\Ext^1(M,N).
 \label{eq:ordered-vanishing}
\end{equation}
\end{proposition}

\begin{proof}
For two preprojectives or two preinjectives, use the reflection-functor
orders in \eqref{eq:enumerations} and \cite[\S4.3]{linxiao}.
The remaining comparisons follow from
\eqref{eq:hom-separation}--\eqref{eq:tube-orthogonality}.
Transitivity follows from the displayed component orders and the order
$\triangleleft$ on the tubes.
\end{proof}
\begin{lemma}\label{regular}
    For a regular representation $M,$ if $N$ is a sub-representation of $M$ and $R$ is a quotient representation of $M,$ with $\partial(N)=0$ and $\partial(R)=0,$ then $N$ and $R$ are regular modules.
\end{lemma}
\begin{proof}
    Since $N$ is a sub-representation of a regualr representation, applying Proposition \ref{homext}, $N$ has no preinjective summands. Similarly, $R$ has no direct preprojective summand.

    Since $\partial(N)=\partial(R)=0,$ they are both regular. 
\end{proof}
\subsection{Multisegments and indexing data}

For $r\geq1$, let $\Pi(r)$ consist of the finite multisegments
\[
 \pi=\sum_{i\in\Z/r\Z,\ l\geq1}m_{i,l}[i;l),
 \qquad m_{i,l}\in\N,
\]
where only finitely many coefficients are nonzero.
For a non-homogeneous tube, abbreviate
$S_{j,i}[l]=S_{z_j,i}[l]$. Under the chosen equivalence for $\cT_j$, set
\[
 M_j(\pi)=\bigoplus_{i\in\Z/r_j\Z,\ l\geq1}
                S_{j,i}[l]^{\oplus m_{i,l}}.
\]
The multisegment is \emph{aperiodic} precisely when
\begin{equation}
 \text{for each }l\geq1\text{ there exists }i\in\Z/r\Z
 \text{ with }m_{i,l}=0.
 \label{eq:aperiodic}
\end{equation}
This is the convention of \cite[\S\S3.1--3.2]{linxiao}.

Let $\N^{(J)}$ denote the finitely supported functions $J\to\N$, and put
\[
 \cG_-=\N^{(\Z_{\leq0})},\qquad
 \cG_+=\N^{(\Z_{>0})},\qquad
 \cG_0=\prod_{j=1}^{e}\Pi(r_j),\qquad
 \cG'=\cG_-\times\cG_0\times\cG_+.
\]
An empty product is a singleton.
For $\bfc=(\bfc_-,\bfc_0,\bfc_+)\in \cG'$ with
$\bfc_0=(\pi_1,\ldots,\pi_e)$, define
\begin{align*}
 c_P&=M(\bfc_-)=\bigoplus_{a\geq0}P_a^{\oplus\bfc_-(-a)},&
 c_I&=M(\bfc_+)=\bigoplus_{b\geq1}I_b^{\oplus\bfc_+(b)},\\
 c_{nh}&=\bigoplus_{j=1}^{e}M_j(\pi_j),&
 c_{nh,j}&=M_j(\pi_j).
\end{align*}
We write $c=M(\bfc)$ for a representative of the isomorphism class of $c_P\oplus c_I\oplus c_{nh}.$ The set of representatives $c$, indexed by $\bfc\in \cG'$, is denoted by $\cG.$

\subsection{The subsets \texorpdfstring{$X(c,m)$ and $Y(c,m)$}{X(c,m) and Y(c,m)}}
\label{subsec:homogeneous-strata}

Fix a representation $c$ with no homogeneous regular direct summand.
In the notation of \eqref{eq:four-parts}, its decomposition is
\begin{equation}
 c\simeq c_{P}\oplus c_{nh}\oplus c_{I},
 \qquad
 c_{nh}\simeq\bigoplus_{j=1}^{e}c_{nh,j},
 \quad c_{nh,j}\in\cT_j.
 \label{eq:C-decomposition}
\end{equation}
There is a unique $\bfc\in\cG'$ such that
$c_{P}\simeq M(\bfc_-)$, $c_{I}\simeq M(\bfc_+)$, and
$c_{nh,j}\simeq M_j(\pi_j)$.

Set $\widetilde{\cG}_\nu
 =\{(c,m)|c \in \cG,\ \ddim c+m\delta=\nu\}.$

Let $m\in\N$ and choose $V$ with $\ddim V=\ddim c+m\delta=\nu$.
Following \cite[\S7.3]{XIAO2023510}, define
\begin{align}
 Y(c,m)
 &=\bigl\{x\in \bfE_\nu:\ M_x\simeq c\oplus R,
       \ R\in\mathbf{R}_{h},\ \ddim R=m\delta\bigr\},
       \label{eq:Y-stratum}\\
 X(c,m)
 &=\bigl\{x\in \bfE_\nu:\ M_x\simeq
       c\oplus R_{u_1}[1]\oplus\cdots\oplus R_{u_m}[1],\notag\\[-2pt]
 &\hspace{38mm}u_1,\ldots,u_m\in Z_{h}
       \text{ pairwise distinct}\bigr\}.
       \label{eq:X-stratum}
\end{align}
Thus $X(c,m)$ consists of the points whose homogeneous regular part is
a direct sum of $m$ pairwise non-isomorphic regular simples.
Here \emph{regular simple} means simple in its tube, not necessarily
simple as a $kQ$-module. Each such summand has dimension vector $\delta$.
Both subsets are $\bfG_\nu$-stable, and $X(c,m)\subseteq Y(c,m)$.
For $m=0$, the empty direct sum is understood, and both sets are the
orbit of $c$.

The Krull--Schmidt theorem gives the finite disjoint decomposition
\begin{equation}
 \bfE_\nu=\bigsqcup_{(c,m)\in \widetilde{\cG}_\nu}Y(c,m);
 \label{eq:geometric-partition}
\end{equation}
see \cite[\S7.3]{XIAO2023510}.

\begin{proposition}
\label{prop:density}
The subset $X(c,m)$ is a smooth variety, open and dense in $Y(c,m)$. In particular,
\begin{equation}
 \overline{X(c,m)}=\overline{Y(c,m)}.
 \label{eq:density}
\end{equation}
\end{proposition}

\begin{proof}
By \cite[6.10]{Lusztig1992}, $X(0,m)$ is open and dense in $\bfE_{m\delta}.$ \[X(c,m)=\left\{(x,F)\middle|\begin{gathered}
    x\in \bfE_{\bfV},F=(0=V_0\subset V_1\subset V_2\subset V_3\subset V_4=\bfV)\\
    x|_{V_1/V_0}\in Y(c_I,0),\, x|_{V_2/V_1}\in Y(c_{nh},0),\, x|_{V_3/V_2}\in X(0,m),\,x|_{V_4/V_3}\in Y(c_P,0)
\end{gathered}\right\}\] and \[Y(c,m)=\left\{(x,F)\middle|\begin{gathered}
    x\in \bfE_{\bfV},F=(0=V_0\subset V_1\subset V_2\subset V_3\subset V_4=\bfV)\\
    x|_{V_1/V_0}\in Y(c_I,0),\, x|_{V_2/V_1}\in Y(c_{nh},0),\, x|_{V_3/V_2}\in Y(0,m),\,x|_{V_4/V_3}\in Y(c_P,0)
\end{gathered}\right\}.\]

Thus $Y(c,m)\subset \overline{X(c,m)}.$ And $X(c,m)$ is open in $Y(c,m).$
Fix a corresponding flag $0\subset V_1\subset V_2\subset V_3\subset \bfV$ and denote the stablizer of it as $P\subset \bfG_{\nu}.$  Let \[X'=\left\{x\middle|\begin{gathered}
    x\in \bfE_{\bfV}\\
    x|_{V_1/V_0}\in Y(c_I,0),\, x|_{V_2/V_1}\in Y(c_{nh},0),\, x|_{V_3/V_2}\in X(0,m),\,x|_{V_4/V_3}\in Y(c_P,0)
\end{gathered}\right\}.\] Then it is clearly that $X'$ is smooth. Thus $X(c,m)=X'\times_P \bfG_{\nu}$ is also smooth.
\end{proof}
Throughout, we use the following notation for the inclusions:
\[
\begin{gathered}
Y(c,m)\xrightarrow{j^{c,m}}\bfE_\nu,\qquad
X(c,m)\xrightarrow{\widetilde j^{c,m}}\bfE_\nu,\qquad X(c,m)\xrightarrow{\widetilde \iota^{c,m}}Y(c,m).
\end{gathered}
\]
\subsection{The partial order in  Xiao--Xu--Zhao \cite{XIAO2023510}}

We now recall \cite[Definition~7.1]{XIAO2023510}, making all auxiliary orders explicit.
For distinct functions, the two lexicographic orders are
\begin{equation}
\begin{aligned}
 \bfc_->_L\bfd_-
 &\ \Longleftrightarrow\
 \bfc_-(t)>\bfd_-(t),\
 t=\max\{u\leq0:\bfc_-(u)\ne\bfd_-(u)\},\\
 \bfc_+>_L\bfd_+
 &\ \Longleftrightarrow\
 \bfc_+(t)>\bfd_+(t),\
 t=\min\{u>0:\bfc_+(u)\ne\bfd_+(u)\}.
\end{aligned}
\label{eq:lex-orders}
\end{equation}
Thus the coordinates are read in the orders $0,-1,-2,\ldots$
and $1,2,3,\ldots$, respectively.
In particular, the latter starts at the $I_1$ end of the preinjective order.

For multisegments of the same cyclic-quiver dimension vector, use
\begin{equation}
 \pi'\leq_G\pi
 \quad\Longleftrightarrow\quad
 \dim_k\Hom(U,M(\pi'))\geq\dim_k\Hom(U,M(\pi)).
 \label{eq:multisegment-order}
\end{equation}
The inequality is required for every $U\in\nilrep_k C_r$;
$M(\pi)$ denotes the corresponding cyclic-quiver representation.
Equivalently, it is enough to test indecomposable $U$.
Multisegments of different cyclic-quiver dimension vectors are not
comparable under $\leq_G$.
This is the order used in \cite[\S4.3]{XIAO2023510}.
For $\bfc_0=(\pi_1,\ldots,\pi_e)$ and
$\bfc'_0=(\pi'_1,\ldots,\pi'_e)$, write
\[
 \bfc'_0<_G\bfc_0
 \quad\Longleftrightarrow\quad
 \pi'_j\leq_G\pi_j\ \text{for all }j,\qquad
 \bfc'_0\ne\bfc_0.
\]

\Needspace{10\baselineskip}
\begin{definition}[{\cite[Definition~7.1]{XIAO2023510}}]
\label{def:xxz-order}
Fix $\nu\in\N^{I}$. For
$(c',m'),(c,m)\in\widetilde{\cG}_\nu$, with corresponding elements $\bfc,\bfc'\in \cG',$ declare
$(c',m')\prec(c,m)$ precisely in the following cases:
\begin{enumerate}[label=\textup{(\alph*)},leftmargin=*]
 \item
 $\bfc'_-\geq_L\bfc_-$ and $\bfc'_+\geq_L\bfc_+$,
 with at least one of these inequalities strict;
 \item
 $\bfc'_-=\bfc_-$, $\bfc'_+=\bfc_+$, and $m'<m$;
 \item
 $\bfc'_-=\bfc_-$, $\bfc'_+=\bfc_+$, $m'=m$,
 and $\bfc'_0<_G\bfc_0$.
\end{enumerate}
Adjoining equality defines $\preceq$ on $\widetilde{\cG}_\nu$.
\end{definition}
\begin{proposition}\cite[Proposition 7.3]{XIAO2023510}\label{open}
    In $\bfE_{\nu},$ $$\overline{X(c,m)}\subset\bigcup_{(c',m')\preceq(c,m)}Y(c',m').$$
\end{proposition}
\begin{remark}\label{ys}
    Since $\overline{X(c,m)}=\overline{Y(c,m)},$ Proposition \ref{open} shows that $Y(c,m)$ is open in $\overline{Y(c,m)}$. It also shows that the subsets $\bigcup_{(c',m')\not\preceq (c,m)} Y(c',m')$ can be used, through successive open and closed decompositions, to construct a filtration of $\bfE_{\nu}.$ Morever, combining with Proposition \ref{prop:density}, $X(c,m)$ is open in $\overline{X(c,m)}.$ And $Y(0,m)$ is dense and open in $\bfE_{m\delta},$ by the same way in Proposition \ref{prop:density}, $Y(c,m)$ is smooth as well. The dimension of $X(c,m)$ is the same as the dimension of $Y(c,m).$
\end{remark}

\section{Lusztig sheaf complexes, the canonical basis, and the PBW basis}
The variety $\bfE_\nu$ has a natural $\bbF_p$-structure with Frobenius map $\operatorname{Fr}$. Lusztig defines $\cD_{m,\bfG_{\nu}}^{b}(\bfE_\nu)$ as the derived category of mixed $\bfG_{\nu}$-equivariant sheaf complexes $(\cD_{m,r,\bfG_{\nu}}^{b}(\bfE_\nu) \text{ with Frobenius map } \operatorname{Fr}^r $
\\$\text{if } q=p^r\bigr).$ The mixed structure is taken with respect to the $\bbF_q$-structure of the variety, which gives rise to a Weil structure on $\cD_{\bfG_{\nu}}^{b}(\bfE_\nu),$ the derived category of $\bfG_{\nu}$-equivariant sheaf complexes. We now recall the induction functor defined by Lusztig:

\[
\Ind_{\nu',\nu''}^{\nu}:
\cD_{m,\bfG_{\nu'}}^{b}(\bfE_{\nu'})\times
\cD_{m,\bfG_{\nu''}}^{b}(\bfE_{\nu''})
\longrightarrow
\cD_{m,\bfG_{\nu}}^{b}(\bfE_\nu),
\]
where the dimension vectors of the $Q_0$-graded vector spaces $\bfV$, $\bfV'$, and $\bfV''$ are $\nu$, $\nu'$, and $\nu''$, respectively. Let
\[
\bfE_\nu'=
\left\{(x,\bfW,\rho_1,\rho_2)\ \middle|\ 
\begin{aligned}
&x\in \bfE_\nu,\ \bfW\cong \bfV'' \text{ as an $Q_0$-graded space},\\
&x(\bfW)\subset \bfW,\,\rho_1:\bfV/\bfW\cong \bfV',\ \rho_2:\bfW\cong \bfV''
\end{aligned}\right\},
\]
\[
\bfE_\nu''=
\left\{(x,\bfW)\ \middle|\ x\in \bfE_\nu,\ x(\bfW)\subset \bfW,\,\bfW\cong \bfV'' \text{ as an $Q_0$-graded space}\right\},
\]
with maps defined by
\[
p_1(x,\bfW,\rho_1,\rho_2)
=
\bigl(\rho_1(x|_{\bfV/\bfW})\rho_1^{-1},\ \rho_2(x|_{\bfW})\rho_2^{-1}\bigr),
\qquad
p_2(x,\bfW,\rho_1,\rho_2)=(x,\bfW),
\qquad
p_3(x,\bfW)=x.
\]
\begin{equation}\label{eq111}
 \begin{tikzcd}
\bfE_{\nu'}\times \bfE_{\nu''} & \bfE_\nu' \arrow[l, "p_1"'] \arrow[r, "p_2"] & \bfE_\nu'' \arrow[r, "p_3"] & \bfE_\nu
\end{tikzcd}
\end{equation}

We denote the dimensions of the fibers of the smooth morphisms $p_1$ and $p_2$ by $d_1$ and $d_2$, respectively. The induction functor is defined as
\[
\Ind_{\nu',\nu''}^{\nu}
:=
{p_3}_! \, {p_2}_{\flat} \, p_1^* [d_1-d_2]\left(\frac{d_1-d_2}{2}\right).
\]
Here ${p_2}_{\flat}$ is the inverse functor to $p_2^*.$
Lusztig then defines flag sheaf complexes as follows. Throughout, we write $\Ind_{\nu',\nu''}^{\nu}(\mathcal{F}\boxtimes \mathcal{G})$ as $\mathcal{F}\diamond \mathcal{G}.$

For $\nu\in \bbN Q_0$ with $\ddim \bfV=\nu$, define the set of flags $S_{\nu}$ by
\[
S_{\nu}=\{\,(\nu_1,\nu_2\cdots \nu_n) \mid\nu_j=l_ji_j,\ i_j\in Q_0,\ \sum_{j=1}^n l_ji_j=\nu\,\}.
\]

For $\underline{\nu}=(\nu_j)\in S_{\nu}$, we consider the flag variety $\widetilde{\mathcal{F}_{\underline{\nu}}}$ defined by
\[
\left\{(x,f)\ \middle|\ 
\begin{aligned}
&x\in \bfE_\nu,\ f=(0=V_n\subset V_{n-1}\subset \cdots \subset V_1\subset V_0=\bfV)\text{ is a filtration of }\bbN Q_0\text{-graded spaces},\\
&x(V_i)\subset V_i,\ \ddim (V_{i-1}/V_i)=\nu_i
\end{aligned}\right\}.
\]
As shown in \cite[Chapter 9]{lusztig2010introduction}, this is a smooth irreducible variety; denote its dimension by $f(\underline{\nu}).$

For $F=(0=V_n\subset V_{n-1}\subset \cdots \subset V_1\subset V_0=\bfV),$ we abbreviate the condition $x(V_i)\subset V_i,$ for every $i$, as $x(F)\subset F.$

There is a natural projective morphism $\pi_{\underline{\nu}}:\widetilde{\mathcal{F}_{\underline{\nu}}}\to \bfE_\nu$. The flag sheaf complex corresponding to $\underline{\nu}$ is $$\mathcal{L}_{\underline{\nu}}:={\pi_{\underline{\nu}}}_!(\overline{\bbQ_l})[f(\underline{\nu})](\frac{f(\underline{\nu})}{2}).$$ By the BBD decomposition theorem \cite{BBD}, it is semisimple. This sheaf complex carries a natural Weil structure. For a flag sheaf complex $\mathcal{L}$ equipped with its natural Weil structure, we retain the notation $\mathcal{L}.$

For any $(c,m)$ and $\underline{\lambda}=(\lambda_1,\cdots,\lambda_l),$ where $\sum_{i=1}^l\lambda_i=m,$ consider the following morphism: $$\begin{aligned}  
&\left\{
(x,F)\ \middle|\
\begin{gathered}
F=(0=V_0\subset V_1=V_{2,0}\subset V_{2,1}\subset\cdots\subset V_{2,e}=V_{3,0}\subset\cdots \subset V_{3,l}\subset V_4=\bfV),\\
 x\in X(c,m),\quad M_x|_{V_{3,i}/V_{3,i-1}}\text{ is homogeneous regular of dimension } \lambda_i\delta,\\x(F)\subset F,\,M_x|_{V_1/V_0}\simeq c_I, M_{x}|_{V_{2,i}/V_{2,i-1}}\simeq c_{nh,i}, M_x|_{V_4/V_{3,m}}\simeq c_P
\end{gathered}
\right\}\xrightarrow{p_{c,\underline{\lambda}}} X(c,m).\end{aligned}$$ It is an etale covering map with coset space $S_{m}/\Pi_{i=1}^lS_{\lambda_i}$ which is the same as in \cite[6.6]{Lusztig1992}, since this property is independent with the choice of $c.$ Thus ${p_{c,\underline{\lambda}}}_!\overline{\bbQ_l}$ is independent of the order of $(\lambda_1,\cdots,\lambda_l).$ We may therefore choose $\underline{\lambda}=(\lambda_1,\cdots,\lambda_l)$ with $\lambda_1\geq \lambda_2\geq\cdots\geq \lambda_l.$ Moreover, $${p_{c,\underline{\lambda}}}_!\overline{\bbQ_l}\simeq \oplus_{\chi\in X(S_m)} L_{\chi}^{\oplus K_{\underline{\lambda},\chi}}$$ with the corresponding Weil structure, where $K_{\underline{\lambda},\chi}$ is a Kostka number.
As in \cite{Lusztig1992}, we write
\[
IC(X(c,m),L_{\chi})={j_{\alpha}}_{!*}\bigl(L_{\chi}[\dim X(c,m)](\frac{\dim X(c,m)}{2})\bigr).
\]
\subsection{Trace map}
Let $X^{\operatorname{Fr}^r}$ and $G^{\operatorname{Fr}^r}$ denote the fixed point sets of $X$ and $G$ under their respective Frobenius maps ${\operatorname{Fr}^r}.$ We denote by $\tilde{\cH}_{G^{\operatorname{Fr}^r}}(X^{\operatorname{Fr}^r})$ the $\bbC$-vector space of $G^{\operatorname{Fr}^r}$-invariant functions on $X^{\operatorname{Fr}^r}$. For any $G$-equivariant morphism $f:X\rightarrow Y$ that respects the $\bbF_q$-structures, we denote its restriction to $X^{\operatorname{Fr}^r}$ by the same letter $f:X^{\operatorname{Fr}^r}\rightarrow Y^{\operatorname{Fr}^r}$. This gives $\bbC$-linear maps
\begin{align*}
f^*:\tilde{\cH}_{G^{\operatorname{Fr}^r}}(Y^{\operatorname{Fr}^r})&\rightarrow \tilde{\cH}_{G^{\operatorname{Fr}^r}}(X^{\operatorname{Fr}^r}) &f_!: \tilde{\cH}_{G^{\operatorname{Fr}^r}}(X^{\operatorname{Fr}^r})&\rightarrow \tilde{\cH}_{G^{\operatorname{Fr}^r}}(Y^{\operatorname{Fr}^r})\\
\varphi&\mapsto (x\mapsto \varphi(f(x))), & \psi&\mapsto(y\mapsto \sum_{x\in f^{-1}(y)}\psi(x)).
\end{align*}

For any mixed Weil complex $A\in \cD^b_{G,r,m}(X)$ with Weil structure $\xi:{\operatorname{Fr}^r}^*(A)\rightarrow A$ and any $x\in X^{\operatorname{Fr}^r}, s\in \bbZ$, there is an isomorphism $H^s(\xi)_x:H^s(A)_x\rightarrow H^s(A)_x$ of the stalk at $x$ of the $s$-th cohomology sheaf. Taking the alternating sum of their traces gives 
$$\chi(A)(x)=\sum_{s\in \bbZ}(-1)^s\tr(H^s(\xi)_x)\in\overline{\bbQ}_l\cong \bbC$$
and hence a function $\chi(A)\in \tilde{\cH}_{G^{\operatorname{Fr}^r}}(X^{\operatorname{Fr}^r}),$ defining a map from $\cD^b_{G,r,m}(X)$ to $\tilde{\cH}_{G^{\operatorname{Fr}^r}}(X^{\operatorname{Fr}^r}),$ called the trace map.

The geometric multiplication on $\oplus_{\nu}\cH_{\bfG_{\nu}^{\operatorname{Fr}^r}}(\bfE_{\nu}^{\operatorname{Fr}^r})$ is defined as follows.
\begin{definition}
 The geometric multiplication $\cH_{\bfG_{\nu'}^{\operatorname{Fr}^r}}(\bfE_{\nu'}^{\operatorname{Fr}^r})\otimes \cH_{\bfG_{\nu''}^{\operatorname{Fr}^r}}(\bfE_{\nu''}^{\operatorname{Fr}^r})\rightarrow \cH_{\bfG_{\nu}^{\operatorname{Fr}^r}}(\bfE_{\nu}^{\operatorname{Fr}^r})$ is defined by $$f\cdot g=\frac{(-\sqrt{q})^{-\sum_{i\in Q_0} \nu'_i\nu''_i-\sum_{h\in H}\nu'_{s(h)}\nu''_{t(h)}}}{|\bfG_{\nu'}^{\operatorname{Fr}^r}\times \bfG_{\nu''}^{\operatorname{Fr}^r}|}{p_3}_!{p_2}_!p_1^*(f\otimes g),$$ where $p_1,$ $p_2,$ and $p_3$ are the maps in Equation~\ref{eq111}.
\end{definition}
\begin{theorem}[{\cite[Theorem 5.3.13]{Pramod-2021}}]\label{sheaf-function correspondence}
For any $G$-equivariant morphism $f:X\rightarrow Y$ that respects the $\bbF_q$-structures and any $A\in \cD^b_{G,r,m}(X), B\in \cD^b_{G,r,m}(Y)$, we have
\begin{align*}
&\chi({A[n]})=(-1)^n\chi(A),\ \chi({A(\frac{n}{2})})=\sqrt{q}^{-n}\chi(A),\ \chi({A\boxtimes B})=\chi(A)\otimes \chi(B),\\
&\chi({f^*(B)})=f^*(\chi(B)),\ \chi({f_!(A)})=f_!(\chi(A)).
\end{align*}
\end{theorem}
 Moreover, $\chi(\mf\diamond\mathcal{G})=\chi(\mathcal{F})\cdot\chi(\mathcal{G}),$ and Lusztig proves in \cite{lusztig1998canonical} that the $\bbZ[\sqrt{q},\sqrt{q}^{-1}]-$span of $\bigsqcup_{\nu}\left\{\chi(\mathcal{L}_{\underline{\nu}})|\underline{\nu}\in S_{\nu}\right\}$ is isomorphic to the $\bbZ-$form of $\bfU_{v}^{-},$ $v=-\sqrt{q},$ with respect to the multiplication $\cdot,$ as explained in detail in \cite{lusztig2010introduction}. 
\subsection{Canonical basis}
In \cite{Lusztig1992}, Lusztig shows that the simple perverse sheaves occurring, up to shift, as direct summands of flag sheaf complexes constitute the sheaf realization of the canonical basis. We denote this set by $\mP=\bigcup_{\nu}\mathcal{P}_{\nu},$ and its elements have the form $IC(X(c,m),L_{\chi}),$ where $c_{nh,j}$ is aperiodic for every $j$ and $\chi\in X(S_m)$ is a character of the symmetric group $S_m.$

In \cite{lusztig1998canonical}, Lusztig defines a Weil structure $\phi_{c,\chi}$ on each simple perverse sheaf $IC(X(c,m),L_{\chi})\in \mP,$ and shows that the resulting trace functions $\chi(IC(X(c,m),L_{\chi}),\phi_{c,\chi})$ form the canonical basis of $\bfU^-_{v=-\sqrt{q}}.$ Actually there is a canonical Weil structure $\phi_{c,\chi}$ on any $IC(X(c,m),L_{\chi})$ induced by the Weil structure on $L_{\chi}|_{X(c,m)},$ for $(c,m)\in \widetilde \cG_{\nu}.$

This Weil structure on $IC(X(c,m),L_{\chi})$ is induced by the following complex: $$\mathcal{L}_{c,(1,\cdots,1)}:=IC(\mo_{c_P},\overline{\bbq_l})\diamond IC(\mo_{c_{nh,1}},\overline{\bbq_l})\diamond \cdots \diamond IC(\mo_{c_{nh,l}},\overline{\bbq_l})\diamond \prod^m IC(X(0,1),\overline{\bbq_l})\diamond IC(\mo_{c_I},\overline{\bbq_l}),$$ of which $IC(X(c,m),\overline{\bbq_l})$ is a direct summand.
\subsection{PBW basis}
Consider ${j^{c,m}}^*\ml_{c,(1,\cdots,1)},$ where $c_{nh,j}$ need not be aperiodic. This complex carries a natural Weil structure $\psi_{c,m},$ whose trace function we compute using the following morphism: $$\begin{aligned}  
&\left\{
(x,F)\ \middle|\
\begin{gathered}
x\in \bfE_\nu,\quad
F=(0=V_0\subset V_1=V_{2,0}\subset V_{2,1}\subset\cdots\subset V_{2,e}=V_{3,0}\subset\cdots \subset V_{3,m}\subset V_4=\bfV),x(F)\subset F,\\
M_x|_{V_1/V_0}\simeq c_I, M_{x}|_{V_{2,i}/V_{2,i-1}}\simeq c_{nh,i}, M_x|_{V_{3,i}/V_{3,i-1}}\text{ is homogeneous regular}, M_x|_{V_4/V_{3,m}}\simeq c_P
\end{gathered}
\right\}\\ &\xrightarrow{\pi_{c,(1,\cdots,1)}} Y(c,m).\end{aligned}$$
\begin{lemma}\label{small}
    $\pi_{c,(1,\cdots,1)}$ is a small morphism respect to $X(c,m)$.
\end{lemma}
\begin{proof}
    The proof is the same as the argument of \cite[Lemma 2.5]{WX2026}. $\pi_{c,(1,\cdots,1)}$ is proper since it could be regarded as a proper morphism $$\begin{aligned}  
&\left\{
(x,F)\ \middle|\
\begin{gathered}
x\in \bfE_\nu,\quad
F=(0=V_0\subset V_1=V_{2,0}\subset V_{2,1}\subset\cdots\subset V_{2,e}=V_{3,0}\subset\cdots \subset V_{3,m}\subset V_4=\bfV),x(F)\subset F,\\
x|_{V_1/V_0}\in \overline{\mo_{c_I}}, {x}|_{V_{2,i}/V_{2,i-1}}\in \overline{\mo_{c_{nh,i}}}, x|_{V_4/V_{3,m}}\in \overline{\mo_{c_P}}
\end{gathered}
\right\}\\ &\rightarrow \bfE_{\nu}\end{aligned}$$ restricting on $Y(c,m).$

    First, $X(c,m)$ is open and dense in $Y(c,m)$ and for any $x\in X(c,m),$ $(\pi_{c,(1,\cdots,1)})^{-1}(x)$ is a set of finite points. Moreover, $$\left\{
(x,F)\ \middle|\
\begin{gathered}
x\in \bfE_\nu,\quad
F=(0=V_0\subset V_1=V_{2,0}\subset V_{2,1}\subset\cdots\subset V_{2,e}=V_{3,0}\subset\cdots \subset V_{3,m}\subset V_4=\bfV),x(F)\subset F,\\
M_x|_{V_1/V_0}\simeq c_I, M_{x}|_{V_{2,i}/V_{2,i-1}}\simeq c_{nh,i}, M_x|_{V_{3,i}/V_{3,i-1}}\text{ is homogeneous regular}, M_x|_{V_4/V_{3,m}}\simeq c_P
\end{gathered}
\right\}$$ is smooth with the same dimension as $\dim X(c,m)$ by the same method as in Proposition \ref{prop:density}.

    Denote \begin{align} 
    Z(c,(t_1,\cdots,t_r)) &=\bigl\{x\in \bfE_\nu:\ M_x\simeq
       c\oplus (R_{u_1}[t_1^1]\oplus\cdots\oplus R_{u_1}[t_1^{j_1}])\oplus\cdots\oplus (R_{u_r}[t_r^1]\oplus\cdots\oplus R_{u_r}[t_r^{j_r}]),\notag\\[-2pt]
 &\hspace{38mm}u_1,\ldots,u_r\in Z_{h}
       \text{ pairwise distinct}\bigr\},\end{align} where $t_i=(t_i^1,\cdots,t_i^{j_i}),$ $t_i^1\leq \cdots\leq t_i^{j_i}\in \bbN_{>0}$ and $\sum_{i=1}^r\sum_{s=1}^{j_i}t_i^s=m.$ It is a locally closed subset by \cite{BONGARTZ2001268}.
       
       Denote $l_i=\sum_{s=1}^{j_i}t_i^s$ and for $x\in Z(c,(t_1,\cdots,t_r)),$ the part $R_{u_i}[t_i^1]\oplus\cdots\oplus R_{u_i}[t_i^{j_i}]$ of $M_x$ is denoted by $M_{x}^{u_i}.$

       $Y(c,m)=\bigsqcup Z(c,(t_1,\cdots,t_r))$ and $Z(c,((1),\cdots,(1)))=X(c,m).$ Fix $x_0\in X(c,m).$

       For $x\in Z(c,(t_1,\cdots,t_r))$ other that $X(c,m),$ as in \cite[Lemma 2.5]{WX2026}, Springer resolution is semismall, $\dim f^{-1}(x)\leq \sum_{i=1}^r\frac{1}{2}(l_i(l_i-1)-(l_i^2-\Hom_{kQ}(M_{x}^{u_i},M_{x}^{u_i}))).$

       Then \[\begin{aligned}
           \dim f^{-1}(x)&\leq \sum_{i=1}^r\frac{1}{2}(l_i(l_i-1)-(l_i^2-\dim\Hom_{kQ}(M_{x}^{u_i},M_{x}^{u_i})))\\
           &= \sum_{i=1}^r\frac{1}{2}(\dim\Hom_{kQ}(M_{x}^{u_i},M_{x}^{u_i})-l_i)\\
           &= \frac{1}{2}((\sum_{i=1}^r \dim\Hom_{kQ}(M_{x}^{u_i},M_{x}^{u_i}) )-m)\\
           &= \frac{1}{2}(\dim\mo_{x_0}-\dim\mo_{x})\\
           &= \frac{1}{2}(\dim X(c,m)-\dim Z(c,(t_1,\cdots,t_r))+r-m).
       \end{aligned}\]

       For the reason that $Z(c,(t_1,\cdots,t_r))\not=X(c,m),$ $r<m.$ Thus $$\dim f^{-1}(x)<\frac{1}{2}(\dim X(c,m)-\dim Z(c,(t_1,\cdots,t_r))).$$

       We have the proof.
\end{proof}
  Thus $${\pi_{c,(1,\cdots,1)}}_!\overline{\bbq_l}[\dim X(c,m)](\frac{\dim(X(c,m))}{2})\simeq{\iota_{c,m}}_{!*}(\oplus_{\chi\in X(S_m)}L_{\chi}^{\oplus K_{(1,\cdots,1),\chi}}[\dim X(c,m)](\frac{\dim(X(c,m))}{2}))$$ with the same Weil structure.
Similarly, for $\underline{\lambda}=(\lambda_1,\cdots,\lambda_l),$ where $\sum_{i=1}^l\lambda_i=m,$ the morphism 
$$\begin{aligned}  
&\left\{
(x,F)\ \middle|\
\begin{gathered}
x\in \bfE_\nu,\quad
F=(0=V_0\subset V_1=V_{2,0}\subset V_{2,1}\subset\cdots\subset V_{2,e}=V_{3,0}\subset\cdots \subset V_{3,l}\subset V_4=\bfV),\\
 M_x|_{V_{3,i}/V_{3,i-1}}\text{ is homogeneous regular of dimension } \lambda_i\delta,\\x(F)\subset F,\,M_x|_{V_1/V_0}\simeq c_I, M_{x}|_{V_{2,i}/V_{2,i-1}}\simeq c_{nh,i}, M_x|_{V_4/V_{3,m}}\simeq c_P
\end{gathered}
\right\}\\ &\xrightarrow{\pi_{c,\underline{\lambda}}} Y(c,m),\end{aligned}$$ is also small respect to $X(c,m)$ and thus ${\pi_{c,\underline{\lambda}}}_!\overline{\bbq_l}$ is independent of the order of $(\lambda_1,\cdots,\lambda_l).$ We choose $\underline{\lambda}=(\lambda_1,\cdots,\lambda_l),$ with $\lambda_1\geq\cdots\geq \lambda_l.$   \begin{equation}\label{eq1}
    {\pi_{c,\underline{\lambda}}}_!\overline{\bbq_l}[\dim X(c,m)](\frac{\dim X(c,m)}{2})\simeq {\iota_{c,m}}_{!*}(\oplus_{\chi\in X(S_m)}L_{\chi}[\dim X(c,m)](\frac{\dim(X(c,m))}{2})^{\oplus K_{\underline{\lambda},\chi}})
\end{equation} with the corresponding Weil structure.

Moreover, ${\iota_{c,m}}_{!*}L_{\chi}[\dim X(c,m)](\frac{\dim(X(c,m))}{2})\simeq {j^{c,m}}^*IC(X(c,m),L_{\chi})$ with the corresponding Weil structure. We denote the Weil structure on ${\iota_{c,m}}_{!*}L_{\chi}[\dim X(c,m)](\frac{\dim(X(c,m))}{2})\simeq {j^{c,m}}^*IC(X(c,m),L_{\chi})$ by $\psi_{c,\chi}={j^{c,m}}^*(\phi_{c,\chi}).$

We now recall the Hall elements of \cite[Sections 6 and 8]{XIAO2023510}
in our notation and compare them with the above trace functions.
Put \(q=p^r\), \(v=-\sqrt q\), and
\(\lvert\beta\rvert=\sum_{i\in Q_0}\beta_i\).
The Hall-algebra formulas below are specialized at this value of \(v\).

\begin{definition}\label{def:hall-N-ct}
Let \(\mathcal H_q(Q)\) be the twisted Hall algebra with basis
\(\{u_{[M]}\}\), indexed by the isomorphism classes of
finite-dimensional \(\mathbb F_qQ\)-modules, and multiplication
\[
 u_{[M]}*u_{[N]}
 =v^{\langle\ddim M,\ddim N\rangle}
   \sum_{[L]}g_{MN}^{L}u_{[L]},
\]
where
\[
 g_{MN}^{L}
 =\#\{\,U\subseteq L:U\simeq N,\ L/U\simeq M\,\},
 \qquad
 \langle\mu,\eta\rangle
 =\sum_{i\in Q_0}\mu_i\eta_i
  -\sum_{h\in H}\mu_{s(h)}\eta_{t(h)}.
\]
For such a module \(M\), set
\[
 \langle M\rangle
 =v^{-\lvert\ddim M\rvert+
          \dim_{\mathbb F_q}\End_{\mathbb F_qQ}(M)}u_{[M]}.
\]
For \(n\geq1\), define
\[
 H_n=v^{-n\lvert\delta\rvert}
       \sum_{\substack{[R],\ R\ {\rm homogeneous\ regular}\\
                        \ddim R=n\delta}}u_{[R]},
 \qquad H_0=1,\qquad H_n=0\quad(n<0).
\]
The sum is over all homogeneous regular
\(\mathbb F_qQ\)-modules of the indicated dimension vector.

For a partition \(\lambda=(\lambda_1,\ldots,\lambda_l)\) of \(m\),
let \(t_\lambda\in X(S_m)\) be the corresponding irreducible character,
and set
\[
 S_\lambda=\det\bigl(H_{\lambda_a-a+b}\bigr)_{1\leq a,b\leq l}.
\]
The determinant is taken in the commutative subalgebra generated by
the \(H_n\), with multiplication \(*\). We set \(S_{\varnothing}=1\).

Let \(c\simeq c_P\oplus c_{nh}\oplus c_I\), with
\(c_{nh}\simeq\bigoplus_{j=1}^{e}c_{nh,j}.\)
Using the corresponding \(\mathbb F_q\)-forms, define
\[
 \begin{aligned}
 N(c,t_\lambda)
 &:=\langle c_P\rangle*
       \langle c_{nh,1}\rangle*\cdots*
       \langle c_{nh,e}\rangle*
       S_\lambda*\langle c_I\rangle\\
 &=\langle c_P\rangle*\langle c_{nh}\rangle*
       S_\lambda*\langle c_I\rangle.
 \end{aligned}
\]
The second equality follows from the vanishing of \(\Hom\) and
\(\Ext^1\) between distinct non-homogeneous tubes.
No aperiodicity assumption on \(c_{nh,j}\) is imposed.
\end{definition}

We identify $\cH_{\bfG_{\nu}^{\operatorname{Fr}^r}}(\bfE_{\nu}^{\operatorname{Fr}^r})$ with the degree-\(\nu\) part of \(\mathcal H_q(Q)\)
by sending the characteristic function
\(\mathbf1_{\mathcal O_M^{\operatorname{Fr}^r}}\) to \(u_{[M]}\).
The geometric multiplication \(\cdot\) and the Hall multiplication
\(*\) are compared in the following lemma.

\begin{lemma}\label{lem:trace-N-ct}
Let \(\lambda\) be a partition of \(m\), put \(\chi=t_\lambda\), and
let \(\nu=\ddim c+m\delta\).
With the Weil structure
\(\psi_{c,\chi}=(j^{c,m})^*\phi_{c,\chi}\), set
\[
 f_{c,\chi}
 :=\chi\bigl((j^{c,m})_!(j^{c,m})^*IC(X(c,m),L_\chi),
                  (j^{c,m})_!\psi_{c,\chi}\bigr).
\]
Under the above vector-space identification, we have
\[
 \boxed{\displaystyle
 f_{c,\chi}
   =v^{-\dim\bfG_\nu+\lvert\nu\rvert}N(c,t_\lambda).}
\]
\end{lemma}

\begin{proof}
For dimension vectors \(\mu,\eta\), write
\[
 b(\mu,\eta)
 =\sum_{i\in Q_0}\mu_i\eta_i+
   \sum_{h\in H}\mu_{s(h)}\eta_{t(h)},
 \qquad
 a_\beta=\dim\bfG_\beta-\lvert\beta\rvert.
\]
The definition of the multiplication on $\oplus_{\nu}\cH_{\bfG_{\nu}^{\operatorname{Fr}^r}}(\bfE_{\nu}^{\operatorname{Fr}^r})$ gives
\[
 \mathbf1_{\mathcal O_M^{\operatorname{Fr}^r}}\cdot\mathbf1_{\mathcal O_N^{\operatorname{Fr}^r}}
 =v^{-b(\mu,\eta)}
    \sum_{[L]}g_{MN}^{L}\mathbf1_{\mathcal O_L^{\operatorname{Fr}^r}},
 \qquad \mu=\ddim M,\quad\eta=\ddim N.
\]
Since \(\dim\bfG_\beta=\sum_i\beta_i^2\), we have
\[
 a_{\mu+\eta}-b(\mu,\eta)
   =a_\mu+a_\eta+\langle\mu,\eta\rangle.
\]
Consequently, the maps
\[
 \Phi_\beta(\mathbf1_{\mathcal O_M^{\operatorname{Fr}^r}})
       =v^{a_\beta}u_{[M]}
\]
satisfying $$\Phi_\beta(\mathbf1_{\mathcal O_M^{\operatorname{Fr}^r}})*\Phi_\alpha(\mathbf1_{\mathcal O_N^{\operatorname{Fr}^r}})=\Phi_{\beta+\alpha}(\mathbf1_{\mathcal O_M^{\operatorname{Fr}^r}}\cdot \mathbf1_{\mathcal O_N^{\operatorname{Fr}^r}}).$$

The equality
\(\dim\mathcal O_M
 =\dim\bfG_{\ddim M}-\dim_{\mathbb F_q}\End_{\mathbb F_qQ}(M)\)
therefore gives
\[
 \Phi_{\ddim M}
 \bigl(v^{-\dim\mathcal O_M}\mathbf1_{\mathcal O_M^{\operatorname{Fr}^r}}\bigr)
 =\langle M\rangle.
\]
Moreover, \(Y(0,n)\) is open in \(E_{n\delta}\), and
\(\dim E_{n\delta}=\dim\bfG_{n\delta}\), since
\(\langle\delta,\delta\rangle=0\). Hence
\[
 \Phi_{n\delta}
 \bigl(v^{-\dim E_{n\delta}}\mathbf1_{Y(0,n)^{\operatorname{Fr}^r}}\bigr)=H_n.
\]

Put \(d=\dim X(c,m)=\dim Y(c,m)\).
For a partition \(\underline\mu=(\mu_1,\ldots,\mu_s)\) of \(m\),
let \(R_{c,\underline\mu}\) be the trace function of
\[
 (j^{c,m})_!\pi_{c,\underline\mu!}
      \overline{\mathbb Q}_l[d](d/2)
\]
with its natural Weil structure.

Applying \(\Phi\), gives
\[
 \begin{aligned}
 \Phi_\nu(R_{c,\underline\mu})
 &=\langle c_P\rangle*\langle c_{nh}\rangle*
      H_{\underline\mu}*\langle c_I\rangle,
 \qquad H_{\underline\mu}:=H_{\mu_1}*\cdots*H_{\mu_s}.
 \end{aligned}
\]

By \eqref{eq1} and the compatibility of the Weil structures,
\[
 R_{c,\underline\mu}
   =\sum_{\lambda\vdash m}K_{\underline\mu,t_\lambda}
                       f_{c,t_\lambda}.
\]
On the other hand, the complete and Schur functions satisfy
\[
 H_{\underline\mu}
   =\sum_{\lambda\vdash m}K_{\underline\mu,t_\lambda}S_\lambda,
\]
and hence
\[
 \Phi_\nu(R_{c,\underline\mu})
   =\sum_{\lambda\vdash m}K_{\underline\mu,t_\lambda}
                         N(c,t_\lambda).
\]
The Kostka matrix is invertible. Comparing the last two expansions
yields \(\Phi_\nu(f_{c,t_\lambda})=N(c,t_\lambda)\).
Since \(\Phi_\nu\) is multiplication by
\(v^{\dim\bfG_\nu-\lvert\nu\rvert}\) under the stated
vector-space identification, the assertion follows.
\end{proof}

\section{Flag sheaf complexes over $Y(c,m)$}
To describe flag sheaf complexes over $Y(c,m)$ more explicitly, we first prove the following lemma.
\begin{lemma}\label{lem:grassmannian-pushforward}
Let $k$ be a field, let $\ell$ be a prime invertible in $k$, and let
$1\leq n\leq m$. Set
\[
  \operatorname{St}_{m,n}
  :=\{A\in\operatorname{Mat}_{m\times n,k}\mid\operatorname{rank}A=n\},
\]
with the right action of $G=\mathbf{G}l_{n,k}$ given by $A\cdot g=Ag$.
Let $\pi:V\to W$ be a $G$-invariant morphism of $k$-varieties.
Assume that there is an \emph{\'etale} covering $\{U_i\to W\}$ and
$G$-equivariant isomorphisms
\[
  V\times_W U_i\simeq U_i\times\operatorname{St}_{m,n}
\]
over $U_i$, where $G$ acts trivially on $U_i$.
Assume that the quotient $X:=V/G$ is represented by a variety, and denote
the induced morphism by $f:X\to W$.
Define
\[
  \mathcal P_{n,m-n}
  :=\{\lambda=(\lambda_1,\ldots,\lambda_n)
       \mid m-n\geq\lambda_1\geq\cdots\geq\lambda_n\geq0\},
  \qquad |\lambda|:=\sum_{a=1}^n\lambda_a.
\]
Then there is an isomorphism in $D^b_c(W,\overline{\bbq_l})$
\begin{equation}\label{eq:grassmannian-pushforward}
  Rf_!\overline{\bbq_l}_X
  \simeq
  \bigoplus_{\lambda\in\mathcal P_{n,m-n}}
  \overline{\bbq_l}_W(-|\lambda|)[-2|\lambda|].
\end{equation}
In particular, if
$b_j:=\#\{\lambda\in\mathcal P_{n,m-n}:|\lambda|=j\}$, then
\[
  R^{2j}f_!\overline{\bbq_l}_X\simeq\overline{\bbq_l}_W(-j)^{\oplus b_j},
  \qquad
  R^{2j+1}f_!\overline{\bbq_l}_X=0.
\]
If $k=\mathbb F_q$, the
isomorphism is compatible with the natural Weil structures.
\end{lemma}

\begin{proof}
The map $A\mapsto\operatorname{im}A$ identifies
$\operatorname{St}_{m,n}\to\operatorname{Gr}(n,m)$ with the frame bundle of the tautological
rank-$n$ subbundle. Consequently, $f$ is \'etale locally isomorphic to
the projection $U_i\times\operatorname{Gr}(n,m)\to U_i$, and $p:V\to X$ is a
principal $G$-bundle. In particular, $f$ is proper, so that $Rf_!=Rf_*$.

Let
\[
  \mathcal S:=V\times^G k^n,
  \qquad (vg,z)\sim(v,gz),
\]
be the associated rank-$n$ vector bundle on $X$. Let
$c_i(\mathcal S)$ denote its \'etale Chern classes, with
$c_0(\mathcal S)=1$, and define the classes $h_r$ by
\[
  \sum_{r\ge0}h_rt^r
  =\left(\sum_{i=0}^{n}c_i(\mathcal S)t^i\right)^{-1},
  \qquad h_r=0\quad(r<0).
\]
For $\lambda\in\mathcal P_{n,m-n}$, set
\[
  \alpha_\lambda
  :=\det(h_{\lambda_a-a+b})_{1\le a,b\le n}
  \in H^{2|\lambda|}_{\mathrm{\acute et}}
       (X,\overline{\bbq_l}(|\lambda|)).
\]
These are the Schur classes $s_\lambda(\mathcal S^\vee)$;
their construction does not require $X$ to be smooth.

For a geometric point $\bar w\to W$, let
\[
  X_{\bar w}:=X\times_W\bar w,
  \qquad i_{\bar w}:X_{\bar w}\to X
\]
be the geometric fiber and its natural morphism to $X$.
The given local trivializations identify $X_{\bar w}$ with
$\operatorname{Gr}(n,m)$ and $i_{\bar w}^*\mathcal S$ with its
tautological subbundle. Writing $\mathcal Q_{\bar w}$ for the
tautological quotient bundle, functoriality of Chern classes
and the Whitney sum formula give
\[
  i_{\bar w}^*h_r=c_r(\mathcal Q_{\bar w})
  \qquad(r\ge0).
\]
Thus the Giambelli formula identifies
$i_{\bar w}^*\alpha_\lambda$ with the \'etale cycle class of
the Schubert variety indexed by $\lambda$; see
\cite[Proposition~14.6.5 and \S14.7]{Fulton1998}
and \cite[Theorem~5.1(1)]{REINEKE200894}.
The Schubert affine paving implies that, for every $d\ge0$,
the classes $i_{\bar w}^*\alpha_\lambda$ with $|\lambda|=d$
form a basis of
$H^{2d}_{\mathrm{\acute et}}
(X_{\bar w},\overline{\bbq_l}(d))$,
and that
$H^{2d+1}_{\mathrm{\acute et}}
(X_{\bar w},\overline{\bbq_l})=0$.

By adjunction, each $\alpha_\lambda$ determines a morphism
\[
  \overline{\bbq_l}_W(-|\lambda|)[-2|\lambda|]
  \longrightarrow Rf_*\overline{\bbq_l}_X.
\]
Taking their direct sum gives
\[
  \Phi:
  \bigoplus_{\lambda\in\mathcal P_{n,m-n}}
  \overline{\bbq_l}_W(-|\lambda|)[-2|\lambda|]
  \longrightarrow Rf_*\overline{\bbq_l}_X.
\]
For every geometric point $\bar w\to W$, proper base change identifies
the target of $\Phi_{\bar w}$ with
$R\Gamma(X_{\bar w},\overline{\bbq_l})$. Here we apply proper base change first with finite \'etale coefficients
and then pass to $\ell$-adic coefficients. For each $d\ge0$, the map induced by $\Phi_{\bar w}$ on
degree-$2d$ cohomology, after the Tate twist $(d)$, is
\[
  \bigoplus_{\substack{\lambda\in\mathcal P_{n,m-n}\\|\lambda|=d}}
  \overline{\bbq_l}
  \longrightarrow
  H^{2d}_{\mathrm{\acute et}}
  (X_{\bar w},\overline{\bbq_l}(d)),
  \qquad
  (a_\lambda)_\lambda
  \longmapsto
  \sum_{|\lambda|=d}a_\lambda\,i_{\bar w}^*\alpha_\lambda.
\]
This is an isomorphism because the classes
$i_{\bar w}^*\alpha_\lambda$ with $|\lambda|=d$ form a basis.
The source and target of $\Phi_{\bar w}$ have vanishing
odd-degree cohomology, so $\Phi_{\bar w}$ is an isomorphism.
Since this holds for every geometric point $\bar w\to W$,
$\Phi$ is an isomorphism in
$D^b_c(W,\overline{\bbq_l})$.
When $k=\mathbb F_q$, the bundle $\mathcal S$ and its Chern classes
are defined over $\mathbb F_q$. The construction of $\Phi$ is therefore
compatible with Frobenius, including the indicated Tate twists.
\end{proof}

\begin{remark}[Required hypotheses]
The proof only requires $f$ to be a proper morphism with geometric
fibers isomorphic to $\operatorname{Gr}(n,m)$, together with a rank-$n$ vector bundle
on $X$ whose restriction to each geometric fiber is the tautological
subbundle under such an isomorphism. The equivariant local
trivializations in Lemma~\ref{lem:grassmannian-pushforward} ensure these
properties. Thus they can be replaced by these two requirements when
that is more convenient in an application.
\end{remark}

For $s\in\tfrac12\mathbb Z$ and $\omega\in\overline{\mathbb Q}_l^{\times}$
with $|\iota_l(\omega)|=1$, let $\mathcal E_{\omega,s}$ be the rank-one
Weil sheaf on $\operatorname{Spec}\mathbb F_q$ on which Frobenius acts
by $\omega q^s$. Thus
$\mathcal E_{\omega,s}=\mathcal E_{\omega,0}(-s)$ in the usual Tate-twist
convention, and it has weight $2s$. In this notation, $\overline{\bbq_l}(s)$ is written as $\mathcal{E}_{1,-s}.$
Here and below, triangulated generation means membership in the full
triangulated subcategory generated by the indicated objects.
\begin{theorem}\label{local}
    For any flag sheaf complex $\mathcal{L},$ any $(c,m)\in \widetilde{\mathcal{G}_{\nu}},$ the restriction $\mathcal{L}|_{Y(c,m)}$ is generated, as a triangulated object, by $IC(X(c,m),L_{\chi})|_{Y(c,m)}\boxtimes \mathcal E_{\omega,s},$ with $|\omega|=1,$ where $\chi\in X(S_m).$
\end{theorem}
\begin{proof}
We claim that for any dimension vector $\tilde{\nu}$, any fixed flag type $(\nu_1,\cdots,\nu_d)$, and any $i<j,$ if $\sum_{s=i+1}^j\nu_s\leq \tilde{\nu},$ then 
\[
\begin{gathered}
\left\{
(x,F)\ \middle|\
\begin{gathered}
x\in \bfE_\nu,\quad
F=(0=V_0\subset V_1\subset\cdots\subset V_d=\bfV),x(F)\subset F,\,\ddim(V_s/V_{s-1})=\nu_s\\
x|_{V_j/V_i}\in Y(c,m),\quad
x|_{V_s/V_{s-1}}\in Y(c_s,m_s),\quad 1\leq s\leq d
\end{gathered}
\right\}
\\
\xrightarrow{\pi}
\left\{
(x,F)\ \middle|\
\begin{gathered}
x\in \bfE_\nu,\quad
F=(0=V_0\subset V_1\subset\cdots\subset V_i\subset V_j\subset V_d=\bfV),x(F)\subset F,\\\ddim(V_s/V_{s-1})=\nu_s,1\leq s\leq i,\,j+1\leq s\leq d,\\
x|_{V_j/V_i}\in Y(c,m),\,x|_{V_s/V_{s-1}}\in Y(c_s,m_s),\quad 1\leq s\leq i,\,j+1\leq s\leq d,
\end{gathered}
\right\},
\end{gathered}
\]
 $\pi_!\overline{\mathbb Q}_l$ is generated by
${\pi_{\underline{\lambda}!}}\overline{\mathbb Q}_l\boxtimes \mathcal E_{\omega,s},$ for suitable $(\omega,s),$ and $\underline{\lambda}=(\lambda_1,\cdots,\lambda_t),$ where $\pi_{\underline{\lambda}}$ is the morphism
\[
\begin{gathered}
\left\{
(x,F)\ \middle|\
\begin{gathered}
x\in \bfE_\nu,\\
F=(0=V_0\subset\cdots\subset V_i\subset V'
=V_{2,0}\subset V_{2,1}\subset\cdots
\subset V_{2,l}=V_{3,0}\subset V_{3,1}
\subset\cdots\subset V_{3,t}\subset V_j\subset\cdots V_d=\bfV),\\
\ddim(V_s/V_{s-1})=\nu_s,\,x(F)\subset F,\,x|_{V'/V_i}\simeq M_{c_I},\quad
x|_{V_{2,s}/V_{2,s-1}}\simeq M_{c_{nh,s}},\\
x|_{V_{3,s}/V_{3,s-1}}
\text{ is homogeneous regular of dimension }\lambda_s\delta,\\
x|_{V_j/V_{3,t}}\simeq M_{c_P}\\
x|_{V_s/V_{s-1}}\in Y(c_s,m_s),\quad 1\leq s\leq i,\,j+1\leq s\leq d,
\end{gathered}
\right\}
\\
\longrightarrow
\left\{
(x,F)\ \middle|\
\begin{gathered}
x\in \bfE_\nu,\quad
F=(0=V_0\subset V_1\subset\cdots\subset V_d=\bfV),x(F)\subset F,\,\ddim(V_s/V_{s-1})=\nu_s\\
x|_{V_j/V_i}\in Y(c,m),\,x|_{V_s/V_{s-1}}\in Y(c_s,m_s),\quad 1\leq s\leq i,\,j+1\leq s\leq d,
\end{gathered}
\right\}.
\end{gathered}
\]
To prove the claim, we will using the induction on $\tilde{\nu}.$ It is clearly that when $\tilde{\nu}=0,$ the proof is trival.

To prove for general $\tilde{\nu},$ first, we consider the following case.

Let $P$ be an indecomposable
preprojective module and $I$ an indecomposable preinjective module, and the dimension of the representation vector space of $P^{\oplus m}\oplus I^{\oplus n}$ is $\tilde{\nu}.$
\[
\begin{aligned}
&
\left\{
(x,F)\ \middle|\
\begin{gathered}
x\in \bfE_\nu,\quad
F=(0=V_0\subset\cdots\subset V_i\subset V_{i+1}\subset V_{i+2}\subset\cdots V_d=\bfV),\ddim(V_s/V_{s-1})=\nu_s\\
x(F)\subset F,\,x|_{V_{i+1}/V_i}\simeq P^{\oplus m},\quad
x|_{V_{i+2}/V_{i+1}}\simeq I^{\oplus n},\\
x|_{V_s/V_{s-1}}\in Y(c_s,m_s),\quad 1\leq s\leq i,\,i+3\leq s\leq d.
\end{gathered}
\right\}
\\
={}&
\left\{
(x,F)\ \middle|\
\begin{gathered}
x\in \bfE_\nu,\quad
F=(0=V_0\subset\cdots\subset V_i\subset V_{i+1}\subset V_{i+2}\subset\cdots V_d=\bfV),\ddim(V_s/V_{s-1})=\nu_s\\
x(F)\subset F,\,x|_{V_{i+2}/V_{i+1}}\simeq I^{\oplus n}\\
x|_{V_s/V_{s-1}}\in Y(c_s,m_s),\quad 1\leq s\leq i,\,i+3\leq s\leq d.
\end{gathered}
\right\}
\\
&-
\bigsqcup_{Y(c',m')\not=Y(P^{\oplus m},0)}
\left\{
(x\in \bfE_\nu,F)\ \middle|\
\begin{gathered}
F=(0=V_0\subset\cdots\subset V_i\subset V_{i+1}\subset V_{i+2}\subset\cdots V_d=\bfV),\ddim(V_s/V_{s-1})=\nu_s\\
x(F)\subset F,\,x|_{V_{i+1}/V_i}\in Y(c',m'),\\
x|_{V_{i+2}/V_{i+1}}\simeq I^{\oplus n}\\
x|_{V_s/V_{s-1}}\in Y(c_s,m_s),\quad 1\leq s\leq i,\,i+3\leq s\leq d.
\end{gathered}
\right\}.
\end{aligned}
\]

Moreover, for any $M$ fitting into the exact sequence
\[
0\longrightarrow P^{\oplus m}
\longrightarrow M
\longrightarrow I^{\oplus n}
\longrightarrow 0,
\]
$M$ has no preinjective direct summand $>I$, and $M$ cannot be of the form $I^{\oplus l}$.

Thus, we only need to consider
\[
\left\{
(x,F)\ \middle|\
\begin{gathered}
x\in \bfE_\nu,\quad
F=(0=V_0\subset\cdots\subset V_i\subset V_{i+1}\subset V_{i+2}\subset\cdots V_d=\bfV),\ddim(V_s/V_{s-1})=\nu_s\\x(F)\subset F,\,\\
x|_{V_{i+2}/V_i}\in Y(c_r,m_r),\quad
x|_{V_{i+2}/V_{i+1}}\simeq I^{\oplus n}\\
x|_{V_s/V_{s-1}}\in Y(c_s,m_s),\quad 1\leq s\leq i,\,i+3\leq s\leq d.
\end{gathered}
\right\}
\] and
\begin{equation}\label{shi}
    \left\{
(x\in \bfE_\nu,F)\ \middle|\
\begin{gathered}
F=(0=V_0\subset\cdots\subset V_i\subset V_{i+1}\subset V_{i+2}\subset\cdots V_d=\bfV),\ddim(V_s/V_{s-1})=\nu_s\\x(F)\subset F,\,\\
x|_{V_{i+2}/V_i}\in Y(c_r,m_r),\quad
x|_{V_{i+1}/V_i}\in Y(c_r',m_r'),\\
x|_{V_{i+2}/V_{i+1}}\simeq I^{\oplus n}\\
x|_{V_s/V_{s-1}}\in Y(c_s,m_s),\quad 1\leq s\leq i,\,i+3\leq s\leq d.
\end{gathered}
\right\},
\end{equation}

where neither $c_{i,I}$ nor $c_{i,I}'$ has a preinjective summand $>I$. Moreover, for each $(c,m)\in \tilde{\cG_{\nu}},$ we have the following: $\forall x\in Y(c,m),$ the dimension $\dim\Hom_{kQ}(M_x,I)$ is independent of the choice of $x\in Y(c,m),$ and we denote this constant by $s_{c,m}.$ Here the mapping rank of an element $f\in \Hom_{kQ}(M_x,I^{\oplus n})$ is the dimension of the vector space spanned by $\{f_1,\cdots,f_n\},$ where $f=(f_1,\cdots,f_n),$ $f_s\in \Hom_{kQ}(M_x,I),$ $s=1,\cdots,n.$ And there is a natural action $\Aut_{kQ}(I^{\oplus n})$ acting on $\Hom_{kQ}(M_x,I^{\oplus n}).$

\[
\begin{aligned}
&
\left\{
(x,F)\ \middle|\
\begin{gathered}
x\in \bfE_\nu,\quad
F=(0=V_0\subset\cdots\subset V_i\subset V_{i+1}\subset V_{i+2}\subset\cdots V_d=\bfV),\ddim(V_s/V_{s-1})=\nu_s\\x(F)\subset F,\,\\
x|_{V_{i+2}/V_i}\in Y(c_r,m_r),\quad
x|_{V_{i+2}/V_{i+1}}\simeq I^{\oplus n}\\
x|_{V_s/V_{s-1}}\in Y(c_s,m_s),\quad 1\leq s\leq i,\,i+3\leq s\leq d.
\end{gathered}
\right\}
\\
={}&
\left\{
(x,F,f)\ \middle|\
\begin{gathered}
x\in \bfE_\nu,\quad
F=(0=V_0\subset\cdots\subset V_i\subset V_{i+2}\subset\cdots V_d=\bfV),\\1\leq s\leq i,\,i+3\leq s\leq d,\ddim(V_s/V_{s-1})=\nu_s,\,x(F)\subset F,\,\\
x|_{V_{i+2}/V_i}\in Y(c_r,m_r),\\
f\in\operatorname{Hom}(M_{x|_{V_{i+2}/V_i}},I^{\oplus n})
\text{ is surjective}\\
x|_{V_s/V_{s-1}}\in Y(c_s,m_s),\quad 1\leq s\leq i,\,i+3\leq s\leq d.
\end{gathered}
\right\}/(\mathrm{GL}_n=\Aut_{kQ}(I^{\oplus n}))
\\
={}&
\left\{
(x,F,f)\ \middle|\
\begin{gathered}
x\in \bfE_\nu,\quad
F=(0=V_0\subset\cdots\subset V_i\subset V_{i+2}\subset\cdots V_d=\bfV),\\1\leq s\leq i,\,i+3\leq s\leq d,\ddim(V_s/V_{s-1})=\nu_s,\,x(F)\subset F,\,\\
x|_{V_{i+2}/V_i}\in Y(c_r,m_r),\\
f\in\operatorname{Hom}(M_{x|_{V_{i+2}/V_i}},I^{\oplus n})
\text{ has mapping rank }n\\
x|_{V_s/V_{s-1}}\in Y(c_s,m_s),\quad 1\leq s\leq i,\,i+3\leq s\leq d.
\end{gathered}
\right\}/\mathrm{GL}_n
\\
&-
\left\{
(x,F,f)\ \middle|\
\begin{gathered}
x\in \bfE_\nu,\quad
F=(0=V_0\subset\cdots\subset V_i\subset V_{i,1}\subset V_{i+2}\subset\cdots V_d=\bfV),\\1\leq s\leq i,\,i+3\leq s\leq d,\ddim(V_s/V_{s-1})=\nu_s,\,x(F)\subset F,\,\\
x|_{V_{i+2}/V_i}\in Y(c_r,m_r),\\
f\in\operatorname{Hom}(M_{x|_{V_{i+2}/V_i}},I^{\oplus n})
\text{ has mapping rank }n,\\
\ker f=V_{i,1}/V_i,\quad
\ddim(V_{i+2}/V_{i,1})<\ddim(I^{\oplus n})\\
x|_{V_s/V_{s-1}}\in Y(c_s,m_s),\quad 1\leq s\leq i,\,i+3\leq s\leq d.
\end{gathered}
\right\}/\mathrm{GL}_n.
\end{aligned}
\]

Fixing the dimension of $\ker f$, we obtain
\[
\bigsqcup_{Y(c_r^1,m_r^1)}
\left\{
(x,F,f)\ \middle|\
\begin{gathered}
x\in \bfE_\nu,\quad
F=(0=V_0\subset\cdots\subset V_i\subset V_{i,1}\subset V_{i+2}\subset\cdots V_d=\bfV),\\1\leq s\leq i,\,i+3\leq s\leq d,\ddim(V_s/V_{s-1})=\nu_s,\,x(F)\subset F,\,\\
f\in\operatorname{Hom}(M_{x|_{V_{i+2}/V_i}},I^{\oplus n})
\text{ has mapping rank }n,\\
\ker f=V_{i,1}/V_i,\quad
x|_{V_{i+2}/V_i}\in Y(c_r,m_r),\quad
x|_{V_{i+2}/V_{i,1}}\in Y(c_r^1,m_r^1)\\
x|_{V_s/V_{s-1}}\in Y(c_s,m_s),\quad 1\leq s\leq i,\,i+3\leq s\leq d.
\end{gathered}
\right\}/\mathrm{GL}_n,
\]
where $c_{r}^1$ has no preinjective direct summand $>I$,
since $M_{x|_{V_{i+2}/V_{i,1}}}$ is a submodule of
$I^{\oplus n}$.

Thus, we only need to consider
\[
\left\{
(x,F,f)\ \middle|\
\begin{gathered}
x\in \bfE_\nu,\quad
F=(0=V_0\subset\cdots\subset V_i\subset V_{i,1}\subset V_{i+2}\subset\cdots V_d=\bfV),\\1\leq s\leq i,\,i+3\leq s\leq d,\ddim(V_s/V_{s-1})=\nu_s,\,x(F)\subset F,\,\\
f\in\operatorname{Hom}(M_{x|_{V_{i+2}/V_{i,1}}},I^{\oplus n})
\text{ has mapping rank }n\text{ and is injective},\\
x|_{V_{i+2}/V_i}\in Y(c_r,m_r),\quad
x|_{V_{i+2}/V_{i,1}}\in Y(c_r^1,m_r^1)\\
x|_{V_s/V_{s-1}}\in Y(c_s,m_s),\quad 1\leq s\leq i,\,i+3\leq s\leq d.
\end{gathered}
\right\}/\mathrm{GL}_n,
\]
which is isomorphic to
\[
\begin{aligned}
&
\left\{
(x,F,f)\ \middle|\
\begin{gathered}
x\in \bfE_\nu,\quad
F=(0=V_0\subset\cdots\subset V_i\subset V_{i,1}\subset V_{i+2}\subset\cdots V_d=\bfV),\\1\leq s\leq i,\,i+3\leq s\leq d,\ddim(V_s/V_{s-1})=\nu_s,\,x(F)\subset F,\,\\
f\in\operatorname{Hom}(M_{x|_{V_{i+2}/V_{i,1}}},I^{\oplus n})
\text{ has mapping rank }n,\\
x|_{V_{i+2}/V_i}\in Y(c_r,m_r),\quad
x|_{V_{i+2}/V_{i,1}}\in Y(c_r^1,m_r^1)\\
x|_{V_s/V_{s-1}}\in Y(c_s,m_s),\quad 1\leq s\leq i,\,i+3\leq s\leq d.
\end{gathered}
\right\}/\mathrm{GL}_n
\\
&-
\bigsqcup_{Y(c_r^2,m_r^2)}
\left\{
(x,F,f)\ \middle|\
\begin{gathered}
x\in \bfE_\nu,\quad
F=(0=V_0\subset\cdots\subset V_i\subset V_{i,1}\subset V_{i,2}\subset V_{i+2}\subset\cdots V_d=\bfV),\\1\leq s\leq i,\,i+3\leq s\leq d,\ddim(V_s/V_{s-1})=\nu_s,\,x(F)\subset F,\,\\
f\in\operatorname{Hom}(M_{x|_{V_{i+2}/V_{i,2}}},I^{\oplus n})
\text{ is injective and has mapping rank }n,\\
x|_{V_{i+2}/V_i}\in Y(c_r,m_r),\quad
x|_{V_{i+2}/V_{i,1}}\in Y(c_r^1,m_r^1),\\
x|_{V_{i+2}/V_{i,2}}\in Y(c_r^2,m_r^2)\\
x|_{V_s/V_{s-1}}\in Y(c_s,m_s),\quad 1\leq s\leq i,\,i+3\leq s\leq d.
\end{gathered}
\right\}/\mathrm{GL}_n.
\end{aligned}
\]

Again, $c_{r}^2$ has no preinjective direct summand $>I$.

By induction, we only need to consider
\[
\left\{
(x,F,f)\ \middle|\
\begin{gathered}
x\in \bfE_\nu,\quad
F=(0=V_0\subset\cdots\subset V_i\subset V_{i,1}\subset V_{i,2}\subset\cdots V_{i,t}\subset V_{i+2}\subset\cdots V_d=\bfV),\\1\leq s\leq i,\,i+3\leq s\leq d,\ddim(V_s/V_{s-1})=\nu_s,\,x(F)\subset F,\,\\
x|_{V_{i+2}/V_i}\in Y(c_r,m_r),\quad
x|_{V_{i+2}/V_{i,1}}\in Y(c_r^1,m_r^1),\quad
\ldots,\quad
x|_{V_{i+2}/V_{i,t}}\in Y(c_r^t,m_r^t),\\
f\in\operatorname{Hom}(M_{x|_{V_{i+2}/V_{i,t}}},I^{\oplus n})
\text{ has mapping rank }n\\
x|_{V_s/V_{s-1}}\in Y(c_s,m_s),\quad 1\leq s\leq i,\,i+3\leq s\leq d.
\end{gathered}
\right\}/\mathrm{GL}_n
\xrightarrow{\pi}\bfE_\nu,
\]
where the representations in $Y(c_r^u,m_r^u)$ have no preinjective direct summand $>I$.
Moreover, $Y(c_r^u,m_r^u)$ is not of type $Y(c_{I^{\oplus k}},0)$. Indeed, $f$ has mapping rank n, and the mapping rank associated with $Y(c_{I^{\oplus k}},0)$ can not be larger than $k<n$ when $u\geq 1.$

By the lemma \ref{lem:grassmannian-pushforward}, for the morphism
\[
\begin{gathered}
\left\{
(x,F,f)\ \middle|\
\begin{gathered}
x\in \bfE_\nu,\quad
F=(0=V_0\subset\cdots\subset V_i\subset V_{i,1}\subset V_{i,2}\subset\cdots V_{i,t}\subset V_{i+2}\subset\cdots V_d=\bfV),\\1\leq s\leq i,\,i+3\leq s\leq d,\ddim(V_s/V_{s-1})=\nu_s,\,x(F)\subset F,\,\\
x|_{V_{i+2}/V_i}\in Y(c_r,m_r),\quad
x|_{V_{i+2}/V_{i,1}}\in Y(c_r^1,m_r^1),\quad
\ldots,\quad
x|_{V_{i+2}/V_{i,t}}\in Y(c_r^t,m_r^t),\\
f\in\operatorname{Hom}(M_{x|_{V_{i+2}/V_{i,t}}},I^{\oplus n})
\text{ has mapping rank }n\\
x|_{V_s/V_{s-1}}\in Y(c_s,m_s),\quad 1\leq s\leq i,\,i+3\leq s\leq d.
\end{gathered}
\right\}/\mathrm{GL}_n
\\
\xrightarrow{\phi}
\left\{
(x,F)\ \middle|\
\begin{gathered}
x\in \bfE_\nu,\quad
F=(0=V_0\subset\cdots\subset V_i\subset V_{i,1}\subset V_{i,2}\subset\cdots V_{i,t}\subset V_{i+2}\subset\cdots V_d=\bfV),\\1\leq s\leq i,\,i+3\leq s\leq d,\ddim(V_s/V_{s-1})=\nu_s,\,x(F)\subset F,\,\\
x|_{V_{i+2}/V_i}\in Y(c_r,m_r),\quad
x|_{V_{i+2}/V_{i,1}}\in Y(c_r^1,m_r^1),\quad
\ldots,\quad
x|_{V_{i+2}/V_{i,t}}\in Y(c_r^t,m_r^t),\\
x|_{V_s/V_{s-1}}\in Y(c_s,m_s),\quad 1\leq s\leq i,\,i+3\leq s\leq d.
\end{gathered}
\right\}
\end{gathered}
\]
we have
\[
\phi_!\overline{\mathbb Q}_l
=
\bigoplus_{\text{for some }k}
\overline{\mathbb Q}_l[2k](k).
\]

Thus, we only need to consider
\[
\left\{
(x,F)\ \middle|\
\begin{gathered}
x\in \bfE_\nu,\quad
F=(0=V_0\subset\cdots\subset V_i\subset V_{i,1}\subset V_{i,2}\subset\cdots V_{i,t}\subset V_{i+2}\subset\cdots V_d=\bfV),\\1\leq s\leq i,\,i+3\leq s\leq d,\ddim(V_s/V_{s-1})=\nu_s,\,x(F)\subset F,\,\\
x|_{V_{i+2}/V_i}\in Y(c_r,m_r),\quad
x|_{V_{i+2}/V_{i,1}}\in Y(c_r^1,m_r^1),\quad
\ldots,\quad
x|_{V_{i+2}/V_{i,t}}\in Y(c_r^t,m_r^t),\\
x|_{V_s/V_{s-1}}\in Y(c_s,m_s),\quad 1\leq s\leq i,\,i+3\leq s\leq d.
\end{gathered}
\right\}.
\]
Actually, in this case, any $M_{x|_{V_{i,l}/V_{i,l-1}}}$ including $M_{x|_{V_{i,1}/V_i}}$ has no preinjective indecomposible direct summand $>I,$ since $M_{x|_{V_{i,l}/V_{i,l-1}}}$ is a submodule of $M_{x|_{V_{i+2}/V_{i,l-1}}}.$

When ${x}|_{V_{i+1}/V_i}\in Y(0,m),$ the same argument applies. In the other cases, when $M_{{x}|_{V_{i+1}/V_i}}$ is nonhomogeneous regular or the direct sum of one indecomposible injective module, only finitely many modules, up to isomorphism, can occur as extensions of $M_{{x}|_{V_{i+1}/V_i}}$ and $I^{\oplus n}$. Moreover, for any $\bfG_{\nu}-$orbit $\mo\subset \bfE_{\nu},$ and any $\mf\in\cD^b_{m,\bfG_{\nu}}(\mo),$ the sheaf $\bfH^i(\mf)$ is generated by sheaves of the form $\mathcal E_{\omega,s},$ for some $\omega,s.$ And if $M_{{x}|_{V_{i+1}/V_i}}<I$
 then the extension of them has no preinjective direct summands $>I.$
All decompositions used here come from filtrations by closed subvarieties. These filtrations are induced by those associated with $\bigcup Y(c,m).$

By induction on the dimension $\tilde{\nu}$, we only need to consider
\[
\left\{
(x,F)\ \middle|\
\begin{gathered}
x\in \bfE_\nu,\\
F=(0=V_0\subset\cdots\subset V_i\subset V_{i,1}\subset V_{i,1}^I\subset V_{i,1}^{nh}\subset V_{i,1}^{hr,1}\subset \cdots V_{i,1}^{hr,t-1}\subset V_{i,1}^P\subset V_{i+2}\subset\cdots V_d=\bfV),\\1\leq s\leq i,\,i+3\leq s\leq d,\ddim(V_s/V_{s-1})=\nu_s,\,x(F)\subset F,\,\\
x|_{V_{i+2}/V_i}\in Y(c_r,m_r),\quad
x|_{V_{i+2}/V_{i,1}^P}\in Y((c_r^1)_P,0),\,x|_{V_{i,1}^I/V_{i,1}}\in Y((c_r^1)_I,0),
\\ \text{denote } V_{i,1}^{hr,0}=V_{i,1}^{nh},\,  V_{i,1}^{hr,t}=V_{i,1}^P
\\x|_{V_{i,1}^{nh}/V_{i,1}^I}\in Y((c_r^1)_{nh},0),\,x|_{V_{i,1}^{hr,l}/V_{i,1}^{hr,l-1}}\in Y(0,\lambda_l), l=1,\cdots,t,\\
x|_{V_s/V_{s-1}}\in Y(c_s,m_s),\quad 1\leq s\leq i,\,i+3\leq s\leq d.
\end{gathered}
\right\},
\] since the case in equation \ref{shi}
\[
\begin{gathered}
    \left\{
(x,F)\ \middle|\
\begin{gathered}
x\in \bfE_\nu,\\
F=(0=V_0\subset\cdots\subset V_i\subset V_{i+1}\subset V_{i+2}\subset\cdots V_d=\bfV),\ddim(V_s/V_{s-1})=\nu_s,\,x(F)\subset F,\,\\
x|_{V_{i+1}/V_i}\in Y(c_r',m_r'),\\
x|_{V_{i+2}/V_{i+1}}\simeq I^{\oplus n}.\,x|_{V_{i+2}/V_i}\in Y(c_r,m_r),\\
x|_{V_s/V_{s-1}}\in Y(c_s,m_s),\quad 1\leq s\leq i,\,i+3\leq s\leq d.
\end{gathered}
\right\}\\\cong\\\left\{
(x,F)\ \middle|\
\begin{gathered}
x\in \bfE_\nu,\\
F=(0=V_0\subset\cdots\subset V_i\subset V_i^I\subset V_i^{nh}\subset V_i^{hr} \subset V_{i+1}\subset V_{i+2}\subset\cdots V_d=\bfV),\\\ddim(V_s/V_{s-1})=\nu_s,\,x(F)\subset F,\,\\
x|_{V_{i+1}/V_i^{hr}}\in Y((c_r')_P,0), x|_{V_i^{hr}/V_i^{nh}}\in Y(0,m_r'),/,x|_{V_i^{nh}/V_i^{I}}\in Y((c_r')_{nh},0),\,x|_{V_i^{I}/V_i}\in Y((c_r')_{I},0)\\
x|_{V_{i+2}/V_{i+1}}\simeq I^{\oplus n}.\,x|_{V_{i+2}/V_i}\in Y(c_r,m_r),\\
x|_{V_s/V_{s-1}}\in Y(c_s,m_s),\quad 1\leq s\leq i,\,i+3\leq s\leq d.
\end{gathered}
\right\}
\end{gathered}
\] can also be reduced to the form like above when applying the claim.

Neither $c_r^1$ nor $c_r'$ has a preinjective direct
summand $>I$, and neither $Y(c_r^1,m_r^1)$ nor $Y(c_r',m_r')$ is of type
$Y(c_{I^{\oplus k}},0)$. After finitely many steps, the rightmost term is therefore replaced by a representation strictly smaller than $I,$ while no injective representation strictly larger than $I$ occurs in the process. There are only finitely many nonisomorphic indecomposable preinjective modules $I'<I,$ with dimension $\leq \tilde{\nu}.$ Thus, after finitely many such steps, all preinjective parts are
on the left.

Moreover, \[
\left\{
(x,F)\ \middle|\
\begin{gathered}
x\in \bfE_\nu,\quad
F=(0=V_0\subset\cdots\subset V_i\subset V_{i+1}\subset V_{i+2}\subset\cdots V_d=\bfV),\\1\leq s\leq i,\,i+3\leq s\leq d,\ddim(V_s/V_{s-1})=\nu_s,\,x(F)\subset F,\,\\
x|_{V_{i+1}/V_i}\in Y(c_{nh,a},0),\quad
x|_{V_{i+2}/V_{i+1}}\in Y(0,m)\\
x|_{V_s/V_{s-1}}\in Y(c_s,m_s),\quad 1\leq s\leq i,\,i+3\leq s\leq d.
\end{gathered}
\right\}
\] is isomorphic to \[
\left\{
(x,F)\ \middle|\
\begin{gathered}
x\in \bfE_\nu,\quad
F=(0=V_0\subset\cdots\subset V_i\subset V_{i+1}\subset V_{i+2}\subset\cdots V_d=\bfV),\\1\leq s\leq i,\,i+3\leq s\leq d,\ddim(V_s/V_{s-1})=\nu_s,\,x(F)\subset F,\,\\
x|_{V_{i+1}/V_i}\in Y(0,m),\quad
x|_{V_{i+2}/V_{i+1}}\in Y(c_{nh,a},0)\\
x|_{V_s/V_{s-1}}\in Y(c_s,m_s),\quad 1\leq s\leq i,\,i+3\leq s\leq d.
\end{gathered}
\right\}
\] and for $i\not=j,$ \[
\left\{
(x,F)\ \middle|\
\begin{gathered}
x\in \bfE_\nu,\quad
F=(0=V_0\subset\cdots\subset V_i\subset V_{i+1}\subset V_{i+2}\subset\cdots V_d=\bfV),\\1\leq s\leq i,\,i+3\leq s\leq d,\ddim(V_s/V_{s-1})=\nu_s,\,x(F)\subset F,\,\\
x|_{V_{i+1}/V_i}\in Y(c_{nh,u},0),\quad
x|_{V_{i+2}/V_{i+1}}\in Y(c_{nh,r},0)\\
x|_{V_s/V_{s-1}}\in Y(c_s,m_s),\quad 1\leq s\leq i,\,i+3\leq s\leq d.
\end{gathered}
\right\}
\] is isomorphic to \[
\left\{
(x,F)\ \middle|\
\begin{gathered}
x\in \bfE_\nu,\quad
F=(0=V_0\subset\cdots\subset V_i\subset V_{i+1}\subset V_{i+2}\subset\cdots V_d=\bfV),\\1\leq s\leq i,\,i+3\leq s\leq d,\ddim(V_s/V_{s-1})=\nu_s,\,x(F)\subset F,\,\\
x|_{V_{i+1}/V_i}\in Y(c_{nh,r},0),\quad
x|_{V_{i+2}/V_{i+1}}\in Y(c_{nh,u},0)\\
x|_{V_s/V_{s-1}}\in Y(c_s,m_s),\quad 1\leq s\leq i,\,i+3\leq s\leq d.
\end{gathered}
\right\}.
\]
The argument for the preprojective part is analogous. 

Hence after finite step, we only need to consider the case, for some $(\lambda_1,\cdots,\lambda_t),$ $\sum_{s=1}^t \lambda_s=m_r,$ \[
\begin{gathered}
\left\{
(x,F)\ \middle|\
\begin{gathered}
x\in \bfE_\nu,\\
F=(0=V_0\subset\cdots\subset V_i\subset V'
=V_{2,0}\subset V_{2,1}\subset\cdots
\subset V_{2,l}=V_{3,0}\subset V_{3,1}
\subset\cdots\subset V_{3,t}\subset V_{i+2}\subset\cdots V_d=\bfV),\\
\ddim(V_s/V_{s-1})=\nu_s,\,x(F)\subset F,\,x|_{V'/V_i}\simeq M_{{c_r}_I},\quad
x|_{V_{2,s}/V_{2,s-1}}\simeq M_{{c_r}_{nh,s}},\\
x|_{V_{3,s}/V_{3,s-1}}
\text{ is homogeneous regular of dimension }\lambda_s\delta,\\
x|_{V_j/V_{3,t}}\simeq M_{{c_r}_P}\\
x|_{V_s/V_{s-1}}\in Y(c_s,m_s),\quad 1\leq s\leq i,\,j+1\leq s\leq d,
\end{gathered}
\right\}\end{gathered},\] which is what we need.

And for general case 
\[
\left\{
(x,F)\ \middle|\
\begin{gathered}
x\in \bfE_\nu,\quad
F=(0=V_0\subset\cdots\subset V_i\subset V_{i+1}\subset V_{i+2}\subset\cdots V_d=\bfV),\\1\leq s\leq i,\,i+3\leq s\leq d,\ddim(V_s/V_{s-1})=\nu_s,\,x(F)\subset F,\,\\
x|_{V_{i+2}/V_i}\in Y(c,m),\quad x|_{V_{i+1}/V_i}\in Y(c',m'),\quad
x|_{V_{i+2}/V_{i+1}}\in Y(c'',m'')\\
x|_{V_s/V_{s-1}}\in Y(c_s,m_s),\quad 1\leq s\leq i,\,i+3\leq s\leq d.
\end{gathered}
\right\},
\] which is isomorphic to \[
\left\{
(x,F)\ \middle|\
\begin{gathered}
x\in \bfE_\nu,\quad
F=(0=V_0\subset\cdots\subset V_i\subset V_i^I \subset V_i^{nh}\subset V_i^{hr}\subset V_{i+1}\subset V_{i+1}^I \subset V_{i+1}^{nh}\subset V_{i+1}^{hr}\subset V_{i+2}\subset\cdots V_d=\bfV),\\1\leq s\leq i,\,i+3\leq s\leq d,\ddim(V_s/V_{s-1})=\nu_s,\,x(F)\subset F,\,\\
x|_{V_{i+2}/V_i}\in Y(c,m),\\
x|_{V_{i+1}/V_i^{hr}}\in Y(c'_P,0),\quad x|_{V_i^{hr}/V_i^{nh}}\in Y(0,m'),\quad x|_{V_i^{nh}/V_i^{I}}\in Y(c'_{nh},0),\quad x|_{V_i^{I}/V_i}\in Y(c'_{I},0),\\
x|_{V_{i+2}/V_{i+1}^{hr}}\in Y(c''_P,0),\quad x|_{V_{i+1}^{hr}/V_{i+1}^{nh}}\in Y(0,m''),\quad x|_{V_{i+1}^{nh}/V_{i+1}^{I}}\in Y(c''_{nh},0),\quad x|_{V_{i+1}^{I}/V_{i+1}}\in Y(c''_{I},0),\\
x|_{V_s/V_{s-1}}\in Y(c_s,m_s),\quad 1\leq s\leq i,\,i+3\leq s\leq d.
\end{gathered}
\right\}.
\]
Then by repeatly applying the claim for lower $\tilde{\nu},$ it will reduced to the case we consider above or the conclusion we need. The proof is similar as before, by considering the order of preprojective and preinjective modules.

The claim is proved.

Together with Equation \ref{eq1}, this completes the proof.
\end{proof}
\begin{remark}\label{rem1}
    Theorem \ref{local} applies in more general settings. For example, it can be used to compute $q_!\overline{\bbQ_l},$ where $$\left\{
(x,F)\ \middle|\
\begin{gathered}
x\in \bfE_\nu,\quad
F=(0= V_1\subset\cdots\subset V_n=\bfV),x(F)\subset F,\,\\
x|_{V_n/V_1}\in Y(c,m),\quad
x|_{V_i/V_{i-1}}\in Y(c_i,m_i),\quad 2\leq i\leq n
\end{gathered}
\right\}\xrightarrow{q} Y(c,m).$$
\end{remark}
\begin{remark}\label{lisse}
    By Theorem \ref{local}, for any flag sheaf complex $\mathcal{L},$ the restriction $\mathcal{L}|_{X(c,m)}$ is generated, as a triangulated object, by the sheaves $L_{\chi}\boxtimes \mathcal E_{\omega,s}$, where $\chi\in X(S_m).$ Thus $\mathbf{H}^i(\mathcal{L}|_{X(c,m)})$ is lisse for every $i$.
\end{remark}
\section{The transition matrix between the canonical basis and the PBW basis}
For any flag sheaf complex $\mathcal{L},$ any $(c,m)\in \widetilde{\mathcal{G}_{\nu}}$ need not to be aperiodic, we establish the following purity statement: $\forall x\in X(c,m),$ the stalk $\mathcal{L}|_x$ is pure.
\begin{theorem}\label{pure}
    For any flag sheaf complex $\mathcal{L},$ any $(c,m)\in \widetilde{\mathcal{G}_{\nu}},$ the restriction $\mathcal{L}|_{X(c,m)}$ is pointwise pure of weight $0.$
\end{theorem}
This theorem will be proved in Section \ref{proof}.
\begin{proposition}\label{prop4.2}
    Fix a flag sheaf complex $\mathcal{L},$ and consider the morphism $Y(c,m)\xrightarrow{j^{c,m}}\bfE_\nu,$ together with the complex ${j^{c,m}}^*\mathcal{L}$. This complex is semisimple and has the following form in $\cD^b_{\bfG_{\nu}}(Y(c,m))$: $$\oplus_{\chi,l}{j^{c,m}}^*IC(X(c,m),L_{\chi})^{\oplus a_{\chi,l}}[l](\frac{l}{2})$$ Choose a basis of the direct sum, for which the Weil structure matrix is upper triangular. The induced Weil structure on each component ${j^{c,m}}^*IC(X(c,m),L_{\chi})[l](\frac{l}{2})$ has the form $\omega\psi_{c,\chi}[l](\frac{l}{2})$ for some $|\omega|=1.$
\end{proposition}
\begin{proof}
    Consider the morphism
\[
Y(c,m)\xrightarrow{j^{c,m}}\bfE_\nu
\]
and a flag sheaf complex $\mathcal L$.
We first prove that ${j^{c,m}}^*\mathcal L$ is pure.

Consider
\[
X(c,m)\xrightarrow{\widetilde j^{c,m}}\bfE_\nu.
\]
Since ${\widetilde j^{c,m*}}\mathcal L$ is pointwise pure of weight 0
and is a lisse sheaf complex on a smooth variety, it is pure by \cite[6.2.5]{PD}.

Moreover, the morphism
\[
\left\{
(x,F)\ \middle|\
\begin{aligned}
&x\in Y(c,m),\\
&F=\bigl(
0=V_0\subset V_1=V_{2,0}
\subset V_{2,1}\subset V_{2,2}\subset\cdots
\subset V_{2,s}=V_{3,0}\\
&\hspace{35mm}
\subset V_{3,1}\subset\cdots\subset V_{3,t}\subset V
\bigr),x(F)\subset F,\,\\
&x|_{V_1/V_0}\cong C_I,\qquad
x|_{V_{2,i}/V_{2,i-1}}\cong C_{nh,i},\\
&x|_{V_{3,j}/V_{3,j-1}}
\text{ is homogeneous regular of dimension } \lambda_i\delta,\\
&x|_{V/V_{3,t}}\cong C_P
\end{aligned}
\right\}
\xrightarrow{\pi_{c,\underline{\lambda}}} Y(c,m)
\]
yields ${\pi_{c,\underline{\lambda}}}_!\overline{\bbq_l}[\dim X(c,m)](\frac{\dim(X(c,m))}{2})\cong {\iota_{c,m}}_{!*}(\oplus_{\chi\in X(S_m)}L_{\chi}[\dim X(c,m)](\frac{\dim(X(c,m))}{2})^{\oplus K_{\underline{\lambda},\chi}})\simeq \oplus_{\chi\in X(S_m)}{j^{c,m}}^*IC(X(c,m),L_{\chi})^{\oplus K_{\underline{\lambda},\chi}}$ as in Equation \ref{eq1}. And ${\iota_{c,m}}_{!*}(L_{\chi}[\dim X(c,m)](\frac{\dim(X(c,m))}{2}))$ is a simple perverse sheaf.

Thus every simple generator of ${j^{c,m}}^*\mathcal L$
has the form $\iota_{c,m!*}P$, up to a shift,
where $P$ is a simple perverse sheaf, up to a shift,
occurring in ${\widetilde{j}^{c,m*}}\mathcal L,$ and $\iota_{c,m}:X(c,m)\rightarrow Y(c,m).$

Moreover, $P$ is a lisse sheaf up to a shift.

As in the proof of Theorem \ref{pure},
we consider the weight filtration of
${}^p\mathbf{H}^i({j^{c,m}}^*\mathcal L)$.

$$0\subset P_1\subset P_2\subset \cdots \subset P_{l-1}\subset P_l= {}^p\mathbf{H}^i({j^{c,m}}^*\mathcal L),$$ where each $P_i/P_{i-1}$ is pure of weight $wt(P_i/P_{i-1})$ and $wt(P_i/P_{i-1})<wt(P_{i+1}/P_i).$

Applying $\iota_{c,m}^*,$ we preserve both the weights and the filtration.

Moreover, ${}^p\bfH^i({\widetilde j^{c,m*}}\mathcal L)\cong \bfH^{i-\dim X(c,m)}({\widetilde j^{c,m*}}\mathcal L)[\dim X(c,m)],$ since $\bfH^{k}({\widetilde j^{c,m*}}\mathcal L)$ is a lisse sheaf complex for any $k.$ Thus ${}^p\bfH^i({\widetilde j^{c,m*}}\mathcal L)\cong \bfH^{i-\dim X(c,m)}({\widetilde j^{c,m*}}\mathcal L)[\dim X(c,m)],$ and this complex is pure of weight $i.$ If $\iota_{c,m}^*(P_i/P_{i-1})=0,$ then $P_i/P_{i-1}=0$, by the description of the simple generators of ${}^p\mathbf{H}^i({j^{c,m}}^*\mathcal L)$ and the faithfulness in \cite[Proposition 5.3.9]{Pramod-2021}.

We find that
${}^p\mathbf{H}^i({j^{c,m}}^*\mathcal L)$
is pure of the same weight as
${}^p\mathbf{H}^i({\widetilde j^{c,m*}}\mathcal L)$.

Thus ${ j^{c,m*}}\mathcal L$ is pure of weight 0
and hence is semisimple. The Weil structure can be computed by restricting it to $X(c,m).$
\end{proof}
\begin{remark}\label{wei1}
    By the proposition above, we have $$\chi({j^{c,m}}^*\mathcal{L})=\sum_{\chi,l}\sum_{i=1}^{a_{\chi,l}}\chi({j^{c,m}}^*IC(X(c,m),L_{\chi}),\omega_{\chi,l}^i\psi_{c,\chi})(-\sqrt{q})^{-l}.$$
\end{remark}
\begin{proposition}\cite[Proposition 8.3]{XIAO2023510}
    There exists a Laurent polynomial $\phi_{\mathcal{L}}^{c,\chi}(v)\in \bbQ[v,v^{-1}],$ for each $(c,m),$ such that, for $q=p^r,$ with $r$ sufficiently large, we have $$\chi(\mathcal{L})=\sum_{c,m}\sum_{\chi\in X(S_m)}\phi_{\mathcal{L}}^{c,\chi}(\sqrt{q})\chi({j^{c,m}}_!{j^{c,m}}^*IC(X(c,m),L_{\chi}),{j^{c,m}}_!\psi_{c,\chi}).$$
\end{proposition}
\begin{proof}
    This proposition follows from $\chi({j^{c,m}}_!{j^{c,m}}^*IC(X(c,m),L_{\chi}),{j^{c,m}}_!\psi_{c,\chi})=(-\sqrt{q})^{-\dim \bfG_{\nu}+\sum_{i\in Q_0} \nu_i}N(c,t_{\chi})$ as defined in \cite[Section 8]{XIAO2023510}, and the fact that $\chi(\mathcal{L})$ is a product of elements of the form $N(c',t'_{\chi}),$ up to multiplication by $v^s,$ for some $s\in \bbZ.$ The result then follows from \cite[Proposition 8.3]{XIAO2023510}.
\end{proof}
\begin{lemma}\cite[Lemma 8.2]{XIAO2023510}
    $\left\{\chi({j^{c,m}}_!{j^{c,m}}^*IC(X(c,m),L_{\chi}),{j^{c,m}}_!\psi_{c,\chi})|(c,m)\in \tilde{\cG}_{\nu},\chi\in X(S_m)\right\}$ is $\bbC-$linearly independent when $q=p^r,$ with $r$ sufficiently large. 
\end{lemma}
\begin{lemma}\label{shi1}
    In Remark \ref{wei1}, $\omega_{\chi,l}^i=1$ for all $\chi,l,i.$ 
\end{lemma}
\begin{proof}
    Using the trace map $\chi(-)$ and applying ${j^{c,m}}^*$ to $\mathcal{L},$ we obtain $$\chi({j^{c,m}}^*\mathcal{L})=\sum_{\chi,l}\sum_{i=1}^{a_{\chi,l}}\chi({j^{c,m}}^*IC(X(c,m),L_{\chi}),\omega_{\chi,l}^i\psi_{c,\chi})(-\sqrt{q})^{-l}$$ and $$\chi({j^{c,m}}^*\mathcal{L})=\sum_{\chi\in X(S_m)}\phi_{\mathcal{L}}^{c,\chi}(\sqrt{q})\chi({j^{c,m}}^*IC(X(c,m),L_{\chi}),\psi_{c,\chi}).$$

    By choosing $\bbF_{q^t}-$points for different $t>0,$ we find that $\sum_{i}(\omega_{\chi,l}^i)^t$ is constant for all sufficiently large $t$. Thus $\omega_{\chi,l}^i=1,$ for all $\chi,l,i.$ 
\end{proof}
For $(c',m')$ which is aperiodic, and any $(c,m)\in \widetilde{\cG_{\nu}},$ $\chi\in X(S_{m}),\chi'\in X(S_{m'}),$ define $f_{c,\chi}^{c',\chi'}(v)\in \bbN[v,v^{-1}]$ by $$f_{c,\chi}^{c',\chi'}(v)=\sum_{i}\dim(\Hom_{\cD^b_{\bfG_{\nu}}(X(c,m))}(\mathbf{H}^{i-\dim X(c,m)}(\tilde{j}^{c,m*}IC(X(c',m'),L_{\chi'})),L_{\chi})v^{i}.$$
\begin{lemma}
 For $(c',m')$ which is aperiodic, and any $(c,m)\in \widetilde{\cG_{\nu}},$ $\chi\in X(S_{m}),\chi'\in X(S_{m'}),$ $f_{c,\chi}^{c',\chi'}(v)$ has the following properties, \\
 $\bullet$ $f_{c,\chi}^{c,\chi}(v)=1,$\\
 $\bullet$ $f_{c,\chi}^{c',\chi'}(v)\in v^{-1}\mathbb{N}[v^{-1}]$ for $(c,m)\prec(c',m')$ and $f_{c,\chi}^{c',\chi'}(v)=0$ for $(c,m)\not\preceq(c',m').$
\end{lemma}
\begin{proof}
 $$\begin{aligned}
     &\dim(\Hom_{\cD^b_{\bfG_{\nu}}(X(c,m))}(\mathbf{H}^{i-\dim X(c,m)}(\tilde{j}^{c,m*}IC(X(c',m'),L_{\chi'})),L_{\chi})\\=&\dim(\Hom_{\cD^b_{\bfG_{\nu}}(Y(c,m))}({}^p\mathbf{H}^{i}(j^{c,m*}IC(X(c',m'),L_{\chi'})),{\iota_{c,m}}_{!*}(L_{\chi}[\dim X(c,m)]))
 \end{aligned}$$
 And $\operatorname{supp}(IC(X(c',m'),L_{\chi'}))=\overline{X(c',m')}\subset \bigcup_{(c,m)\preceq(c',m')}Y(c,m)$ by Proposition \ref{open}. Thus $f_{c,\chi}^{c',\chi'}(v)=0$ for $(c,m)\not\preceq(c',m').$

 By \cite[Lemma 3.1.4, Lemma 3.3.4]{Pramod-2021}, for $(c,m)\prec(c',m'),$ ${j^{c,m}}^*IC(X(c',m'),L_{\chi'})\in {}^p\cD^{\leq -1}_{\bfG_{\nu}}(Y(c,m)).$ Then we have the proof.
\end{proof}
\begin{theorem}\label{the4.7}
    For $(c',m')$ which is aperiodic, and any $(c,m)\in \widetilde{\cG_{\nu}},$ $\chi\in X(S_{m}),\chi'\in X(S_{m'}),$ we have the identity $$\chi(IC(X(c',m'),L_{\chi'}),\phi_{c',\chi'})=\sum_{c,\chi} f_{c,\chi}^{c',\chi'}(-\sqrt{q})\chi({j^{c,m}}_!{j^{c,m}}^*IC(X(c,m),L_{\chi}),{j^{c,m}}_!\psi_{c,\chi}).$$
\end{theorem}
\begin{proof}
    The object $(IC(X(c',m'),L_{\chi'}),\phi_{c',\chi'})[n](\frac{n}{2})$ is a direct summand of some flag sheaf complex $\mathcal{L}$. Choose a direct sum decomposition of $\mathcal{L},$ for which the Weil structure matrix is upper triangular. The diagonal block corresponding to this summand is $\phi_{c',\chi'}[n](\frac{n}{2}).$

    As proved in Lemma \ref{shi1}, we have $$\chi({j^{c,m}}^*\mathcal{L})=\sum_{\chi,l}\sum_{i=1}^{a_{\chi,l}}\chi({j^{c,m}}^*IC(X(c,m),L_{\chi}),\psi_{c,\chi})(-\sqrt{q})^{-l}.$$ Since $(IC(X(c',m'),L_{\chi'}),\phi_{c',\chi'})[n](\frac{n}{2})$ is a direct summand of $\mathcal{L},$ we have $${j^{c,m}}^*IC(X(c',m'),L_{\chi'})\cong \sum_{\chi,l}{j^{c,m}}^*IC(X(c,m),L_{\chi})^{\oplus b_{\chi,l}}[l](\frac{l}{2}).$$ The induced Weil structure gives $$\begin{aligned}
  \chi({j^{c,m}}^*IC(X(c',m'),L_{\chi'}),{j^{c,m}}^*\phi_{c',\chi'})&=\sum_{\chi,l}b_{\chi,l}\chi({j^{c,m}}^*IC(X(c,m),L_{\chi})[l](\frac{l}{2}),\psi_{c,\chi}[l](\frac{l}{2}))\\
  &=\sum_{\chi,l}b_{\chi,l}\chi({j^{c,m}}^*IC(X(c,m),L_{\chi}),\psi_{c,\chi})(-\sqrt{q}^{-1})^l
    \end{aligned}$$
    Combining this with $$\chi(IC(X(c',m'),L_{\chi'}),\phi_{c',\chi'})=\sum_{c,m}j^{c,m}_!{j^{c,m}}^*\chi(IC(X(c',m'),L_{\chi'}),\phi_{c',\chi'})$$

    completes the proof.
\end{proof}
\begin{remark}
Let $A$ be the set of aperiodic indices and $P$ its complement in a fixed degree.
By~\cite[Lemma~9.13]{XIAO2023510}, the PBW elements satisfy
\[
 E_\alpha=N_\alpha+\sum_{\pi\in P}b_\alpha^\pi(v)N_\pi,
 \qquad \alpha\in A.
\]
For $\beta\in A,$ $B_{\beta}$ is the corresponding canonical basis $(-\sqrt{q})^{\dim \bfG_{\nu}-\sum_{i\in Q_0}\nu_i}\chi(IC_{\beta},\phi_{\beta}).$ Write $B_\beta=\sum_{\alpha\in A}g_\alpha^\beta(v)E_\alpha$.
Comparing the coefficients of $N_\alpha$, for $\alpha\in A$, with the expansion
given by Theorem~\ref{the4.7} yields $g_\alpha^\beta(v)=f_\alpha^\beta(v)$.
Consequently, the aperiodic part of the geometric coefficient formula gives the
transition matrix from the canonical basis to the PBW basis.\end{remark}
\section{Proof of Theorem \ref{pure}}\label{proof}
To prove purity, we first establish the following lemmas.
\begin{lemma}\label{simp}
    Let $\mP$ be a simple perverse sheaf on a variety $X$. For a locally closed subvariety $U\xrightarrow{j}X,$ if $U\cap \operatorname{supp}(\mP)$ is open and nonempty in $\operatorname{supp}(\mP),$ then $\mP\simeq j_{!*}j^*\mP.$
\end{lemma}
\begin{proof}
    Write the support of $\mP$ as the closed subvariety $V\xrightarrow{i}X.$ Then $i_!i^*\mP\simeq \mP$ and $i_!i^!\mP\simeq \mP.$ It follows that $i^*\mP\simeq i^!\mP.$ Thus, by \cite[Lemma 3.1.4]{Pramod-2021}, $i^*\mP\simeq i^!\mP$ is perverse. Given an exact sequence of perverse sheaves $$0\rightarrow \mP_1\rightarrow i^*\mP\rightarrow \mP_2\rightarrow 0,$$ applying $i_!,$ we obtain the exact sequence of perverse sheaves$$0\rightarrow i_!\mP_1\rightarrow \mP\rightarrow i_!\mP_2\rightarrow 0.$$ Since $\mP$ is a simple perverse sheaf, either $i_!\mP_1=0$ or $i_!\mP_2=0.$ Applying $i^*,$ we obtain $\mP_1=0$ or $\mP_2=0.$ Thus $i^*\mP$ is a simple perverse sheaf.

    Set $U'=U\cap V$, and let $j':U'\rightarrow V$ be the open embedding. Since the support of $i^*\mP$ is $V,$ we have $j'^*i^*\mP\not=0$. Moreover, $j'^*i^*\mP$ is perverse because $j'$ is an open embedding. Then by \cite[Lemma 3.3.8]{Pramod-2021}, $i^*\mP\simeq j'_{!*}j'^*i^*\mP.$ This gives $$\mP\simeq i_!j'_{!*}j'^*i^*\mP.$$ Denote the closed embedding $U\cap V\rightarrow U$ by $i'.$ By the definition of the functor $j'_{!*},$ $$i_!j'_{!*}j'^*i^*\mP\simeq j_{!*}i'_!i'^*j^*\mP.$$ Finally, $\operatorname{supp}(j^*\mP)\subset V\cap U$ implies that $i'_!i'^*j^*\mP\simeq j^*\mP.$

    This completes the proof.
\end{proof}
\begin{lemma}[Pointwise purity for filtrations in a regular tube]\label{lem:tube-filtration-affine-paving}
Let $k=\overline{\mathbb F}_q$, let $Q$ be a connected acyclic quiver
of affine type with vertex set $Q_0$, and let $\delta$ be its minimal
positive imaginary root. Fix an $Q_0$-graded vector space $V$ and
$x_0\in \bfE_\bfV$ such that $M_{x_0}$ belongs to a non-homogeneous regular
tube $\cT$. Let $m\geq1$, let $J\subseteq\{1,\ldots,m\}$, and let
$\boldsymbol\gamma=(\gamma_1,\ldots,\gamma_m)$ satisfy
\[
  \sum_{a=1}^m\gamma_a=\dim V,
  \qquad \gamma_a=\delta\quad(a\notin J).
\]
Let $\mathcal{F}_{x_0,\boldsymbol\gamma,J}$ be the reduced variety of
$Q_0$-graded flags
\[
  F=(0=V^0\subseteq V^1\subseteq\cdots\subseteq V^m=V)
\]
such that $x_0V^a\subseteq V^a$, $\dim(V^a/V^{a-1})=\gamma_a$ for
all $a$, and $M_{x_0}|_{V^a/V^{a-1}}$ is semisimple regular for
$a\in J$. Here semisimple regular means a direct sum of quasi-simple
regular representations.

Then, for $\ell\ne\operatorname{char}k$, the complex
\[
  R\Gamma(\mathcal{F}_{x_0,\boldsymbol\gamma,J},\overline{\bbq_l})
\]
on a point is pure of weight $0$.
The empty variety is allowed.
\end{lemma}

\begin{proof}
We may assume that $\mathcal{F}_{x_0,\boldsymbol\gamma,J}$ is nonempty.
We may also enlarge the finite field of definition so that all the
finite-dimensional data and morphisms constructed below are defined
over it. This replaces Frobenius by a positive power and does not
change the assertion about weights.

Every $\gamma_a$ has defect zero. If $U\subseteq M_{x_0}$ has defect
zero, then $U$ has no preinjective direct summand, by the
Hom-vanishing between preinjective and regular representations.
Additivity of the defect therefore excludes preprojective summands,
so $U$ is regular. Similarly, $M_{x_0}/U$ has no preprojective
summand and has defect zero, hence is regular. Orthogonality of
distinct tubes implies that both objects belong to $\cT$.
Consequently, all terms and successive quotients of the flags under
consideration belong to $\cT$.

Let $r$ be the rank of $\cT$, and let $Q'$ be the cyclic quiver with
vertex set $Q_0'=\mathbb Z/r\mathbb Z$ and arrows $u\longrightarrow u+1$.
Fix an exact $k$-linear equivalence
\[
  \mathscr E:\rep^{\mathrm{nil}}(kQ')\xrightarrow{\sim}\cT,
\]
as in \cite[Section 6.3]{linxiao}, and write
$\mathscr E(M_{x'_0})\simeq M_{x_0}$, where $x'_0$ is nilpotent on
the $I'$-graded space $\bfW=\bigoplus_{u\in Q_0'}\bfW_u$.
For $u\in Q_0'$, let $S_u$ be the image of the vertex simple and put
$\mathbf d_u=\dim S_u$. These vectors satisfy
\[
  \sum_{u\in Q_0'}\mathbf d_u=\delta,
  \qquad
  \langle\mathbf d_u,\mathbf d_v\rangle
      =\delta_{uv}-\delta_{v,u+1}.
\]
They are linearly independent over $\mathbb Q$: a relation
$\sum_u c_u\mathbf d_u=0$ forces $c_v=c_{v-1}$ for every $v$, and
their common value is zero since $\delta\ne0$.
Thus the prescribed dimensions determine unique vectors
$\eta_a\in\mathbb N^{I'}$ such that
\[
  \gamma_a=\sum_{u\in Q_0'}\eta_{a,u}\mathbf d_u.
\]
Their existence follows from nonemptiness of the filtration variety.
In particular, $\eta_a=(1,\ldots,1)$ for $a\notin J$.
Put $\boldsymbol\eta=(\eta_1,\ldots,\eta_m)$.

The equivalence identifies $\mathcal{F}_{x_0,\boldsymbol\gamma,J}$ with
the reduced variety
\begin{equation}\label{eq:notes-cyclic-fiber}
  \left\{
  \begin{array}{l|l}
    F=(0=W^0\subseteq\cdots\subseteq W^m=\bfW)
    &\begin{array}{l}
       \dim(W^a/W^{a-1})=\eta_a,\quad x'_0W^a\subseteq W^a
           \quad(1\leq a\leq m),\\
       x'_0W^a\subseteq W^{a-1}\quad(a\in J)
     \end{array}
  \end{array}
  \right\}.
\end{equation}

Set
\[
  \bfE_\bfW=\bigoplus_{u\in Q_0'}\Hom_k(\bfW_u,\bfW_{u+1}),\qquad
  G_\bfW=\prod_{u\in Q_0'}\GL(\bfW_u),\qquad
  \mathfrak g_\bfW=\bigoplus_{u\in Q_0'}\End_k(\bfW_u).
\]
The case $\bfW=0$ is immediate. Choose a decomposition
$M_{x'_0}=\bigoplus_{i=1}^s M_i$ into indecomposable representations,
with homogeneous Jordan bases
\[
  e_{i,1},\ldots,e_{i,l_i},\qquad
  x'_0e_{i,j}=e_{i,j-1},\qquad e_{i,0}=0.
\]
Write $d_i=\deg e_{i,1}\in Q_0'$, so that
$\deg e_{i,j}=d_i+1-j$. Define
\begin{equation}\label{eq:notes-lambda}
  w_{i,j}=(s+1)(l_i+1-2j)+i-1,
  \qquad \lambda(t)e_{i,j}=t^{w_{i,j}}e_{i,j}.
\end{equation}
Each $e_{i,j}$ is homogeneous, hence $\lambda:G_m\to G_\bfW$ is a
homomorphism of algebraic groups. For $2\leq j\leq l_i$ we have
$w_{i,j-1}-w_{i,j}=2(s+1)$, while $x'_0e_{i,1}=0$. Hence
\begin{equation}\label{eq:notes-x-weight}
  \lambda(t)x'_0\lambda(t)^{-1}=t^{2(s+1)}x'_0.
\end{equation}

We next construct a complement to the orbit tangent space.
Consider
\[
  \psi:\mathfrak g_\bfW\longrightarrow \bfE_\bfW,
  \qquad \psi(h)=hx'_0-x'_0h.
\]
The stabilizer $\Aut_{kQ'}(M_{x'_0})$ is the open set of invertible
elements in $\End_{kQ'}(M_{x'_0})$, and is therefore smooth.
It follows that $T_{x'_0}(G_\bfW x'_0)=\im\psi$.
Let $\widetilde\psi$ denote the same commutator map on
$\End_k(W)$. For $\alpha=(i,j)$ and $\beta=(i',j')$, define the
matrix unit $E_{\alpha,\beta}$ by
\[
  E_{\alpha,\beta}(e_{i',j'})=e_{i,j},
  \qquad E_{\alpha,\beta}(e_{h,b})=0\quad\text{if }(h,b)\ne(i',j').
\]
Thus the first index specifies the output vector and the second
the input vector; equivalently, we use column vectors throughout.
With $\alpha-1=(i,j-1)$ and $\beta+1=(i',j'+1)$, direct computation gives
\begin{equation}\label{eq:notes-commutator}
  \widetilde\psi(E_{\alpha,\beta})
   =E_{\alpha,\beta+1}-E_{\alpha-1,\beta},
\end{equation}
where a matrix unit with an out-of-range position is zero.

For a fixed ordered pair $(i,i')$, put $A=l_i$ and $B=l_{i'}$.
The commutator preserves the rectangular block of maps from the
$i'$th chain to the $i$th chain. In its cokernel, equation
\eqref{eq:notes-commutator} gives
\[
  [E_{(i,j),(i',j')}]
    =[E_{(i,j+1),(i',j'+1)}]
    \quad(j<A,\ j'<B).
\]

The surviving classes are therefore spanned by
$E_{(i,A),(i',b)}$, $1\leq b\leq\min(A,B)$.
They are linearly independent. 
Taking all ordered pairs proves
\[
  \End_k(\bfW)=\im\widetilde\psi\oplus\widetilde C,
  \qquad
  \widetilde C=\operatorname{Span}\{E_{(i,l_i),(i',j')}:
        1\leq i,i'\leq s,\ 1\leq j'\leq\min(l_i,l_{i'})\}.
\]

Give $\End_k(\bfW)$ its $I'$-grading, where a map of degree $d$ sends
$W_u$ into $W_{u+d}$. Every chosen matrix unit is homogeneous,
and $\widetilde\psi$ increases degree by one. In particular, if
$h=\sum_d h_d$ is its homogeneous decomposition, then the degree-one
part of $\widetilde\psi(h)$ is $\psi(h_0)$. Taking the degree-one
part of the preceding direct sum therefore gives
\begin{equation}\label{eq:notes-complement}
  \bfE_\bfW=\im\psi\oplus C,
  \qquad C=\widetilde C\cap \bfE_\bfW.
\end{equation}
Explicitly,
\[
  C=\operatorname{Span}\left\{E_{(i,l_i),(i',j')}:
     \begin{array}{l}
       1\leq i,i'\leq s,\quad 1\leq j'\leq\min(l_i,l_{i'}),\\
       \deg e_{i,l_i}=\deg e_{i',j'}+1\text{ in }I'
     \end{array}\right\}.
\]
In particular, pairs with $i=i'$ are included.
Put $C'=x'_0+C$. The notation $C$ will always denote the linear
complement, and $C'$ the affine slice.

Define an action on the full representation space by
\begin{equation}\label{eq:notes-rho}
  \rho(t)y=t^{2(s+1)}\lambda(t)^{-1}y\lambda(t).
\end{equation}
Equation \eqref{eq:notes-x-weight} gives $\rho(t)x'_0=x'_0$.
For a retained basis vector of $C$, its weight is
\begin{align*}
  2(s+1)+w_{i',j'}-w_{i,l_i}
    &=(s+1)(l_i+l_{i'}+2-2j')+i'-i\\
    &\geq 2(s+1)-(s-1)=s+3>0.
\end{align*}
Thus $C$ is $\rho$-stable, and the action on $C'$ extends to
$\mathbb{A}^1\times C'\to C'$ with $\rho(0)y=x'_0$ for all $y\in C'$.
In particular, it contracts $C'$ to $x'_0$.

Consider the action map
\begin{equation}\label{eq:notes-Phi}
  \Phi:G_\bfW\times C'\longrightarrow \bfE_\bfW,
  \qquad (g,y)\longmapsto gyg^{-1}.
\end{equation}
Its differential at $(1,x'_0)$ is $(h,c)\mapsto\psi(h)+c$,
which is surjective by \eqref{eq:notes-complement}. Since source
and target are smooth, $\Phi$ is smooth at $(1,x'_0)$.
For the action
\[
  t\cdot(g,y)
     =(\lambda(t)^{-1}g\lambda(t),\rho(t)y),
\]
$\Phi$ is equivariant with respect to $\rho$ on $\bfE_\bfW$.
Consequently,
\[
  \mathcal B=\{y\in C':\Phi\text{ is not smooth at }(1,y)\}
\]
is closed, being the inverse image of the non-smooth locus under
$y\mapsto(1,y)$, and is $\rho$-stable by equivariance.
If $\mathcal B$ were nonempty, contraction
would imply $x'_0\in\mathcal B$, a contradiction. Left translation on
$G_\bfW$, together with conjugation on $\bfE_\bfW$, then shows that $\Phi$
is smooth everywhere. Moreover,
$\dim C'=\dim \bfE_\bfW-\dim(\mo_{ x'_0})$ by
\eqref{eq:notes-complement}, so $C'$ is a transverse slice at $x'_0$.

We now consider the incidence projection with no prescribed flag
type. Let
\[
  \mathscr A_m(W)=
    \left\{\boldsymbol\xi=(\xi_1,\ldots,\xi_m)
       \in(\mathbb N^{I'})^m:\sum_a\xi_a=\dim \bfW\right\},
\]
and let $\mathcal{F}_{\boldsymbol\xi}(\bfW)$ be the variety of $I'$-graded flags
$F=(0=W^0\subseteq\cdots\subseteq W^m=\bfW)$ with successive
dimension vectors $\boldsymbol\xi$. $\mathcal{F}_{\boldsymbol\xi}(\bfW)$ is a smooth projective variety. Define
\begin{equation}\label{eq:notes-incidence}
  \widetilde {\mathcal{F}_{\boldsymbol\xi,J}}
   =\left\{(y,F)\in \bfE_\bfW\times\mathcal{F}_{\boldsymbol\xi}(\bfW):
       \begin{array}{l}
         yW^a\subseteq W^a\quad(1\leq a\leq m),\\
         yW^a\subseteq W^{a-1}\quad(a\in J)
       \end{array}\right\},
  \qquad \pi(y,F)=y.
\end{equation}
The stated containments are closed conditions on the universal
subbundles. Hence $ \widetilde {\mathcal{F}_{\boldsymbol\xi,J}}$ is closed in
$\bfE_\bfW\times\mathcal{F}_{\boldsymbol\xi}(\bfW)$, and $\pi:\widetilde {\mathcal{F}_{\boldsymbol\xi,J}}\to \bfE_\bfW$ is
proper. Fix a filtration $0=W^0\subset\cdots \subset W^m=\bfW$ and let $P\subset G_{\bfW}$ be the stablizer of this filtration. It is easy to see that $\widetilde {\mathcal{F}_{\boldsymbol\xi,J}}=G_{\bfW}\times_P \widetilde {F_{\boldsymbol\xi,J}},$ where \begin{equation*}
  \widetilde {F_{\boldsymbol\xi,J}}
   =\left\{y\in \bfE_\bfW:
       \begin{array}{l}
         yW^a\subseteq W^a\quad(1\leq a\leq m),\\
         yW^a\subseteq W^{a-1}\quad(a\in J)
       \end{array}\right\}.
\end{equation*} $\widetilde {F_{\boldsymbol\xi,J}}$ is smooth thus $\widetilde {\mathcal{F}_{\boldsymbol\xi,J}}$ is smooth.

Set
\[
  \cK=R\pi_!\overline{\bbq_l}{}_{\widetilde {\mathcal{F}_{\boldsymbol\xi,J}}}
       =R\pi_*\overline{\bbq_l}{}_{\widetilde {\mathcal{F}_{\boldsymbol\xi,J}}}.
\]
The constant
complex with its natural Weil structure is pure of weight $0$ on smooth variety.
Proper pushforward therefore makes $\cK$ pure of
weight $0$ in the mixed derived category
\cite{PD}; see also
\cite[Example 5.4.5 and Corollary 5.4.9]{Pramod-2021}.
The incidence projection is $G_\bfW$-equivariant.
It is also equivariant for $\rho$, using the action
\[
  t\cdot(y,F)=(\rho(t)y,\lambda(t)^{-1}F)
\]
on $\widetilde {\mathcal{F}_{\boldsymbol\xi,J}}$. Indeed,
$(\rho(t)y)(\lambda(t)^{-1}W^a)
=t^{2(s+1)}\lambda(t)^{-1}(yW^a)$, which preserves both
containment conditions in \eqref{eq:notes-incidence}.
The constant complex is equivariant, and proper base change
therefore gives the $G_\bfW$-equivariance of $\cK$ and its weak
$G_m$-equivariance for $\rho$; cf.\ \cite[Exercise 2.10.2]{Pramod-2021}.
These equivariance isomorphisms are induced from the natural Weil
structure on the constant complex and are Frobenius-compatible.

Let $\iota:C'\hookrightarrow \bfE_\bfW$ be the inclusion. We verify that $\iota^*\cK$ is pure of weight $0$,
using the argument in the proof of
\cite[Proposition 5.7.3]{Pramod-2021} for the pure complex $\cK$.

Let $C'\xrightarrow{i} G_\bfW\times C'$, $y\mapsto(1,y).$ Then $\iota^*\cK=i^*\Phi^*\cK.$ And since that $\Phi^*\cK$ is $G_\bfW-$equivariant as well as $\cK,$ $$i^*[-\dim G_\bfW]\operatorname{For}^{G_\bfW}\Phi^*\cK\cong i^![\dim G_\bfW](\dim G_\bfW)\operatorname{For}^{G_\bfW}\Phi^*\cK$$ by \cite[Theorem 6.5.10]{Pramod-2021}. $\Phi$ is smooth, and denote the relative dimension of map $\Phi$ as $d.$ Then $$\Phi^*\cK[2d](d)\cong \Phi^!\cK.$$ Together with \cite[Theorem 5.4.7]{Pramod-2021}, $\Phi^*\cK$ is pure of weight $0$ and $i^*\Phi^*\cK$ is pure of weight $0.$ Thus $\iota^*\cK$ is pure of weight $0.$

The inclusion $\iota$ is $G_m$-equivariant, hence $\iota^*\cK$ is weakly
equivariant for the contracting action on $C'$.
Let $j_0:\{x'_0\}\hookrightarrow C'$ be the inclusion and let
$a_0:C'\to\{x'_0\}$ be the structure map. The contraction
isomorphism gives
\begin{equation}\label{eq:notes-contraction}
  R(a_0)_*\iota^*\cK\simeq j_0^*\iota^*\cK;
\end{equation}
see \cite[Theorem 2.10.3]{Pramod-2021}, for the mixed $\ell$-adic formulation.
All maps are defined over the chosen finite field, so this
isomorphism is compatible with Frobenius. Since $\iota^*\cK$ is pure
of weight $0$, the left-hand side of
\eqref{eq:notes-contraction} has weights at least $0$, whereas
the right-hand side has weights at most $0$. Both sides are
therefore pure of weight $0$.

For $i_{x'_0}:\{x'_0\}\hookrightarrow \bfE_\bfW$, 
\[
  i_{x'_0}^*\cK
    =j_0^*\iota^*\cK
    \simeq R\Gamma(\pi^{-1}(x'_0),\overline{\bbq_l}),
\]
which is pure of weight $0$. Thus $R\pi_!\overline{\bbq_l}$ is pointwise pure
of weight $0$ at every nilpotent point. 
The identification of $\pi^{-1}(x_0')$ with
$\mathcal{F}_{x_0,\boldsymbol\gamma,J}$ proves the assertion.
\end{proof}

We now prove Theorem \ref{pure}: for any $x\in X(c,m)$ and any flag sheaf complex $\mathcal{L}$,
the stalk $\mathcal{L}_x$ is pure. To prove purity, we may extend the field of definition. Given an $\bbF_{q^l}-$point $x,$ we may regard it as an $\bbF_{q^{l'}}-$point, with $l'\geq l.$ We may therefore choose $x$ as an $\bbF_{q^{l'}}-$point so that $x$ decomposes as $x_p\oplus x_I\oplus\oplus_jx_{nh,j}\oplus_ix_{hr,i}$ over the field $\bbF_{q^{l'}}.$
\begin{proof}
For any $(\mathbf{c},\mathbf{m}),$ fix $x\in X(\mathbf{c},\mathbf{m})$. We can choose $V_I,V_P,V_{nh,i},V_{hr,j}$ such that
\[
x=x|_{V_I}\oplus x|_{V_P}
  \oplus\bigoplus_i x|_{V_{nh,i}}
  \oplus\bigoplus_j x|_{V_{hr,j}},
\]
and $M_x|_{V_I}\cong \mathbf{c}_I$, $M_x|_{V_P}\cong \mathbf{c}_P$,
$M_x|_{V_{nh,i}}\cong \mathbf{c}_{nh,i}$, and
$M_x|_{V_{hr,j}}$ is a simple homogeneous regular module. Denote $x_I$ as $x|_{V_I}$ and $x_P$ as $x|_{V_P}.$

 We make the contracting action explicit. Refine the preprojective and
preinjective parts into their isotypic summands:
\[
 V_P=\bigoplus_{a\in A_P}V_{P,a},\qquad
 V_I=\bigoplus_{b\in A_I}V_{I,b},
\]
where
\[
 M_{x|V_{P,a}}\simeq P_a^{\oplus m_a},\qquad
 M_{x|V_{I,b}}\simeq I_b^{\oplus n_b}.
\]
Here $A_P$ and $A_I$ are finite sets of indices, with the conventions
$P_0<P_1<\cdots$ and $\cdots<I_2<I_1$ of Section~1.
Choose an integer
\[
 N>\max\bigl(A_P\cup A_I\cup\{0\}\bigr).
\]
Define a cocharacter $\lambda:\mathbb G_m\longrightarrow G_\nu$ by
\[
\begin{aligned}
 \lambda(t)|_{V_{P,a}}&=t^{N-a}\operatorname{id}_{V_{P,a}},&
 \lambda(t)|_{V_{I,b}}&=t^{b-N}\operatorname{id}_{V_{I,b}},\\
 \lambda(t)|_{V_{\mathrm{nh},i}}&=\operatorname{id}_{V_{\mathrm{nh},i}},&
 \lambda(t)|_{V_{\mathrm{hr},j}}&=\operatorname{id}_{V_{\mathrm{hr},j}}.
\end{aligned}
\]
Since $x$ is block diagonal, $\lambda(t)\cdot x=x$.

1. We construct transversal slices in $\bfE_{V_{nh,i}},\bfE_{V_{hr,j}}$.

We construct complements to the tangent spaces of $\mo_{nh,j}$ and $\mo_{hr,j}$, denoted by $\mathcal{T}_{nh,i},\mathcal{T}_{hr,j}$, respectively.

Using the notation of Section \ref{sec1}, for any non-homogeneous regular module $M_i(\pi)$ in the tube $\cT_i,$ the condition $\Hom_kQ(M_x,M_i(\pi))>0$ is closed in $\bfE_{\nu}.$ It follows that the condition that $M_x$ is regular and has no non-homogeneous direct summands outside $\cT_j$ is open.

Choose a corresponding open subset with $x|_{V_{nh,i}}\in U_{nh,i}$,
such that, for every $x\in U_{nh,i}$, the representation $x$ is regular and has no non-homogeneous direct summands in other tubes, and
\[
\mathrm{G}_{V_{nh,i}}\times U_{nh,i}\longrightarrow \bfE_{V_{nh,i}}
\quad\text{is smooth}.
\] 
 Clearly $U_{nh,i}$ is smooth.
 
We can also choose transversal slices $\mathcal{T}_P$ and $\mathcal{T}_I$ at the points $x_P$ and $x_I$ in $\bfE_{V_{P}}$ and $\bfE_{V_{I}}.$ By \cite[Proposition 5.7.3]{Pramod-2021}, $G_{m}$ defined above actions on $\mathcal{T}_P$ and $\mathcal{T}_I$ with fixed points $x_P$ and $x_I$, respectively. These $G_m$ actions
extend to $\mathbb A$-actions.

Set $\bfE_{V,V'}:=\oplus_{h\in H}\Hom_k(V_{s(h)},V'_{t(h)}).$

We can choose a complement to the tangent space of $\mo_x$ at $x$, stable under $G_m\subset\bfG_{\nu}$,
as a subvariety $\mathcal{T}$ of
\[
\begin{aligned}
&\prod\mathcal{T}_{nh,i}\times\prod\mathcal{T}_{hr,j}
 \times\prod_j \bfE_{V_I,V_{hr,j}}\times\prod_i \bfE_{V_I,V_{nh,i}}\\
&\qquad\times\prod_j \bfE_{V_{hr,j},V_P}
 \times\prod_i \bfE_{V_{nh,i},V_P}\times \bfE_{V_I,V_P}\times \mathcal{T}_P\times \mathcal{T}_I.
\end{aligned}
\]
 $\mathbb{A}\supset G_m$ acts trivially on $\prod\mathcal{T}_{nh,i}\times\prod\mathcal{T}_{hr,j}\times x_P\times x_I$ and $\lambda(t)$ acts on $$\prod_j \bfE_{V_I,V_{hr,j}}\times\prod_i \bfE_{V_I,V_{nh,i}}\times\prod_j \bfE_{V_{hr,j},V_P}
 \times\prod_i \bfE_{V_{nh,i},V_P}\times \bfE_{V_I,V_P}\times \mathcal{T}_P\times \mathcal{T}_I,$$ satisfying $\lambda(0)x=\lim_{t\rightarrow0}(\lambda(t)(x'))={0}\times x_P\times x_I,$ for any $x',$ and $G_m\subset \bfG_{\nu}.$

We can choose an open subvariety $U_{hr}\subset\prod\mathcal{T}_{hr,j}$
such that $U_{hr}\subset X(0,\mathbf{m})$ and denote $\mathcal{T}_{hr,j}\cap X(0,1)$ as $U_{hr,j}.$ The subvariety
\[
\begin{aligned}
&\prod U_{nh,i}\times U_{hr}
 \times\prod_j \bfE_{V_I,V_{hr,j}}\times\prod_i \bfE_{V_I,V_{nh,i}}\times \mathcal{T}_P\times \mathcal{T}_I\\
&\qquad\times\prod_j \bfE_{V_{hr,j},V_P}
 \times\prod_i \bfE_{V_{nh,i},V_P}\times \bfE_{V_I,V_P}\cap\mathcal{T}
\end{aligned}
\]
is open in $\mathcal{T}$ and stable under $G_m\subset\bfG_{\nu}$, where the $G_m$ action
extends to an $\mathbb A$-action.

We can also choose an open subvariety $U\subset\mathcal{T}$ such that,
for any $x\in U$, $\Hom(M_x|_{V_{nh,i}},M_x|_{V_{nh,j}})=0$ for $i\ne j$ and $\Hom(M_x|_{V_{nh,i}},M_x|_{V_{hr,k}})=0,$
and such that this subvariety is stable under $G_m$ and
\[
\bfG_{\nu}\times U\longrightarrow \bfE_\nu
\quad\text{is a smooth morphism}.
\]
Finally, we choose
\[
\begin{aligned}
\widetilde U=U\cap\Bigl(
 &\bigl(U\cap(x_P\times x_{I}\times\prod U_{nh,i}\times U_{hr}\times\{0\})\bigr)
 \times\prod_j \bfE_{V_I,V_{hr,j}}\times\prod_i \bfE_{V_I,V_{nh,i}}\\
 &\times\prod_j \bfE_{V_{hr,j},V_P}
 \times\prod_i \bfE_{V_{nh,i},V_P}\times \bfE_{V_I,V_P}\times \mathcal{T}_P\times \mathcal{T}_I\Bigr)
\end{aligned}
\]
which has the $G_m$ action that extends to an $\mathbb A$-action and is stable under Frobenius morphism $\operatorname{Fr}_{q^{l'}}.$
Moreover, $U\cap(x_P\times x_{I}\times\prod U_{nh,i}\times U_{hr}\times\{0\})$ is open in
$x_P\times x_{I}\times\prod U_{nh,i}\times U_{hr}\times\{0\}$.

$\widetilde U$ is open in $U$.
Using the morphisms
\[
\widetilde U\xrightarrow{i}\bfG_{\nu}\times\widetilde U
\xrightarrow{\phi}\bfE_\nu,
\]
where $\phi$ is smooth, together with the isomorphism
\[
i^*\operatorname{For}^{\bfG_{\nu}}[-\dim\bfG_{\nu}]
\cong i^!\operatorname{For}^{\bfG_{\nu}}[\dim\bfG_{\nu}](\dim\bfG_{\nu})
\]
we obtain
\[
\Rightarrow\quad \operatorname{For}^{\bfG_{\nu}}i^*\phi^*\mathcal{L}\quad\text{is still pure}.
\]
Let
\[
U_0=U\cap(x_P\times x_I\times\prod U_{nh,j}\times U_{hr}\times\{0\})
\xrightarrow{u_0}\bfE_\nu
\]
be the $G_m$ fixed point set. Clearly, it is smooth.
By \cite[Theorem 2.10.3]{Pramod-2021},
$u_0^*\operatorname{For}^{\bfG_{\nu}}\mathcal{L}$ is pure.
In fact, for any pure $\bfG_{\nu}-$equivariant sheaf complex $F$,
$u_0^*\operatorname{For}^{\bfG_{\nu}} F$ is pure. The complex
\[
u_0^*\operatorname{For}^{\bfG_{\nu}}\mathcal{L}=\mathcal{L}_1\oplus\mathcal{L}_2\oplus\cdots\oplus\mathcal{L}_n,
\]
is semisimple, with its summands ordered by support. We will show the order and what $L_i$ is in the followings. 

Equip $U_0$ with the stratification
$\bigsqcup_{(c,m)}U_0\cap Y(c,m)=U_0$,
and write $U_0\cap Y(c,m)=U_0^{c,m}$,
$U_0\cap X(c,m)=\widetilde U_0^{c,m}$. It could be equiped with an order $\bigsqcup_{i=1}^nU_0^{c_i,m_i},$ such that $U_0^{c_s,m_s}$ is open in $\bigsqcup_{i=s}^nU_0^{c_i,m_i},$ for any $s=1,\cdots,n.$ 

 For any $(c,m)\in \tilde{\cG}_{\nu},$ if $Y(c,m)\cap U_0\not=\emptyset,$ then $c_P\cong x_P,$ and $c_I\cong x_I.$  

$u_0^*\operatorname{For}^{\bfG_{\nu}}\mathcal{L}$ is semisimple thus could be written as $\oplus_l S_l,$ where $S_l$ is a shift of simple perverse sheaf. Denote $C_i=\{l|\operatorname{supp}(S_l)\cap (\cup_{s=1}^{i-1}U_0^{c_s,m_s})=\emptyset, \operatorname{supp}(S_l)\cap U_0^{c_i,m_i}\not=\emptyset\}.$ Let $L_i=\oplus_{l\in C_i}S_l.$ It is clearly that $U_0^{c_i,m_i}\cap \operatorname{supp}(L_i)$ is open in $\operatorname{supp}(L_i).$

Use the notation
\[
\begin{gathered}
U_0^{c,m}\xrightarrow{j_0^{c,m}}\bfE_\nu,\qquad
\widetilde U_0^{c,m}\xrightarrow{\widetilde j_0^{c,m}}\bfE_\nu,\\
Y(c,m)\xrightarrow{j^{c,m}}\bfE_\nu,\qquad
X(c,m)\xrightarrow{\widetilde j^{c,m}}\bfE_\nu.
\end{gathered}
\]
Write $a^{c,m}:U_0^{c,m}\hookrightarrow U_0$ for the internal inclusion.

Each restriction functor $j_0^{c,m*}$ can be expressed using the composition
\[
U_0^{c,m}\xrightarrow{i^{c,m}}Y(c,m)\xrightarrow{j^{c,m}}\bfE_\nu.
\]

For $\mathcal{F}$ satisfying Theorem \ref{local}, the restriction $j_0^{c,m*}\mf$ is generated by $i^{c,m*}(IC(X(c,m),L_{\chi})|_{Y(c,m)})\simeq j_0^{c,m*}IC(X(c,m),L_{\chi})$ in $\cD^b(U_0^{c,m}).$ 
Moreover, $u_0^*IC(X(c,m),L_{\chi})$ is semisimple in $\cD^b(U_0),$ since $u_0^*$ preserves purity. We have $\operatorname{supp}(u_0^*IC(X(c,m),L_{\chi}))\subset \overline{Y(c,m)}\cap U_0$. Since $U_0^{c,m}$ is open in $\overline{Y(c,m)}\cap U_0,$ the restriction $j_0^{c,m*}IC(X(c,m),L_{\chi})$ is semisimple in $\cD^b(U_0^{c,m}).$ If $j_0^{c,m*}\mf$ is semisimple in $\cD^b(U_0^{c,m}),$ each of its direct summands occurs as a direct summand of $j_0^{c,m*}IC(X(c,m),L_{\chi})$ for some $\chi.$

Moreover, 
\[
a^{c,m*}u_0^* \Ind\left(
 IC(O_{x_P},\overline{\bbQ_l})\boxtimes\prod IC(X(0,1),\overline{\bbQ_l})
\boxtimes\prod IC(O_{c_{nh,j}},\overline{\bbQ_l})\boxtimes IC(O_{x_I},\overline{\bbQ_l})
\right)
\quad\text{is still pure of weight }0,
\]
since $U_0^{c,m}$ is open in $U_0\cap \operatorname{supp}(\Ind\left(
 IC(O_{x_P},\overline{\bbQ_l})\boxtimes\prod IC(X(0,1),\overline{\bbQ_l})
\boxtimes\prod IC(O_{c_{nh,j}},\overline{\bbQ_l})\boxtimes IC(O_{x_I},\overline{\bbQ_l})
\right))$
and is therefore semisimple in $\cD^b(U^{c,m}_0).$
Each simple perverse sheaf occurring above, up to a shift, has the form $a^{c,m*}P,$ where $P$ is a simple perverse direct summand, up to a shift, of $$u_0^* \Ind\left(
 IC(O_{x_P},\overline{\bbQ_l})\boxtimes\prod IC(X(0,1),\overline{\bbQ_l})
\boxtimes\prod IC(O_{c_{nh,j}},\overline{\bbQ_l})\boxtimes IC(O_{x_I},\overline{\bbQ_l})
\right).$$

We now consider the case $(c_1,m_1).$ The map $a^{c_1,m_1}$ is an open embedding.
It follows that ${j_0^{c_1,m_1}}^*\mathcal L$ is semisimple and
that its simple summands occur in $${a^{c_1,m_1}}^*u_0^* \Ind\left(
 IC(O_{x_P},\overline{\bbQ_l})\boxtimes\prod IC(X(0,1),\overline{\bbQ_l})
\boxtimes\prod IC(O_{c_{nh,j}^{1}},\overline{\bbQ_l})\boxtimes IC(O_{x_I},\overline{\bbQ_l})
\right).$$
We can therefore choose a basis consisting of shifts of simple perverse sheaves in $L_1,$ with respect to which the matrix of the Weil structure is upper triangular. Moreover, let $P$ be a simple perverse direct summand, up to a shift, of $L_1,$ such that ${a^{c_1,m_1}}^*P$ is not $0$. This restriction occurs in $${a^{c_1,m_1}}^*u_0^* \Ind\left(
 IC(O_{x_P},\overline{\bbQ_l})\boxtimes\prod IC(X(0,1),\overline{\bbQ_l})
\boxtimes\prod IC(O_{c_{nh,j}^{1}},\overline{\bbQ_l})\boxtimes IC(O_{x_I},\overline{\bbQ_l})
\right),$$ up to shift.

By Lemma \ref{simp}, $P$ is a direct summand(up to shift) of $$u_0^* \Ind\left(
 IC(O_{x_P},\overline{\bbQ_l})\boxtimes\prod IC(X(0,1),\overline{\bbQ_l})
\boxtimes\prod IC(O_{c_{nh,j}^{1}},\overline{\bbQ_l})\boxtimes IC(O_{x_I},\overline{\bbQ_l})
\right),$$

which is a direct summand, up to a shift and twist $[s](s/2)$ for some $s$, of
\[
u_0^*\Ind\left(
 IC(O_{x_P},\overline{\bbQ_l})\boxtimes\prod IC(X(0,1),\overline{\bbQ_l})
\boxtimes\prod IC(O_{nh,\mathrm{semi}},\overline{\bbQ_l})\boxtimes IC(O_{x_I},\overline{\bbQ_l})
\right)
\]

where $O_{nh,\mathrm{semi}}$ is the orbit of a semisimple non-homogeneous module
in a single tube. The argument uses the assertion that
$IC(O_{c_{nh,j}^{1}},\overline{\mathbb Q}_l)[s](s/2)$ occurs as a direct
summand of an iterated induction of such semisimple-orbit IC complexes by \cite[Section~6.9]{Lusztig1992}, \cite[proof of Proposition~2.7]{StroppelWebster2011} and \cite[Corollary~3.2]{DengDuXiao2007} .
The operation $[s](s/2)$ preserves purity and weight.
On a Frobenius-stable simple constituent, two pure Weil normalizations
of the same weight differ by a scalar of absolute value $1$. By the semisimplicity of $u_0^*\ml,$ $$u_0^*\Ind\left(
 IC(O_{x_P},\overline{\bbQ_l})\boxtimes\prod IC(X(0,1),\overline{\bbQ_l})
\boxtimes\prod IC(O_{c_{nh,j}^{1}},\overline{\bbQ_l})\boxtimes IC(O_{x_I},\overline{\bbQ_l})
\right)$$ and $$u_0^*\Ind\left(
 IC(O_{x_P},\overline{\bbQ_l})\boxtimes\prod IC(X(0,1),\overline{\bbQ_l})
\boxtimes\prod IC(O_{nh,\mathrm{semi}},\overline{\bbQ_l})\boxtimes IC(O_{x_I},\overline{\bbQ_l})
\right),$$ we can choose a basis consisting of shifts of simple perverse sheaves, ordered so that the matrix of the Weil structure is upper triangular. All the complexes have weight $0,$ so the Weil structures on the same shifted simple perverse sheaf in these complexes differ only by a scalar of absolute value $1.$ Thus, if for any $(c,m),$ $$\Ind\left(
 IC(O_{x_P},\overline{\bbQ_l})\boxtimes\prod IC(X(0,1),\overline{\bbQ_l})
\boxtimes\prod IC(O_{nh,\mathrm{semi}},\overline{\bbQ_l})\boxtimes IC(O_{x_I},\overline{\bbQ_l})
\right)$$ is pointwise pure on $U_0\cap X(c,m),$ then $L_1$ with the induced Weil structure is pointwise pure on $U_0\cap X(c,m).$  

More precisely, since $L_1$ is semisimple, we can choose a basis consisting of shifts of simple perverse sheaves with respect to which the matrix of its Weil structure is upper triangular as a restriction on $u_0^*\ml.$ It suffices to prove pointwise purity for the Weil structure restricted to each such shifts of simple perverse sheaf, which occuring in \[
u_0^*\Ind\left(
 IC(O_{x_P},\overline{\bbQ_l})\boxtimes\prod IC(X(0,1),\overline{\bbQ_l})
\boxtimes\prod IC(O_{nh,\mathrm{semi}},\overline{\bbQ_l})\boxtimes IC(O_{x_I},\overline{\bbQ_l})
\right),
\]  up to shift and tate twist $[l](\frac{l}{2}).$
This structure differs from the corresponding Weil structure only by a scalar of absolute value $1$. 

Indeed, for any $x'\in U_0\cap X(c',m')$ for any $(c',m')$, let $i_{x'}:x'\hookrightarrow \bfE_\nu$ be the inclusion. Then
\[
i_{x'}^*\Ind\left(
 IC(O_{x_P},\overline{\bbQ_l})\boxtimes\prod IC(X(0,1),\overline{\bbQ_l})
\boxtimes\prod IC(O_{nh,\mathrm{semi}},\overline{\bbQ_l})\boxtimes IC(O_{x_I},\overline{\bbQ_l})
\right)
\]
is pure of weight $0.$ The proof is following.

By Lemma \ref{regular}, in this case we only need to consider the regular part. 
Any regular module with dimension $\delta$ satisfies all direct summands of it in the same tube and $\overline{\mo_{nh,semi}}\cap \mo_M=\emptyset,$ if $O_M$ is the orbit of any regular module other than this $\mo_{nh,semi}.$ Then apply Lemma \ref{lem:tube-filtration-affine-paving} and deal with multiple nonhomogeneous tubes by \cite[Corollary 5.4.8]{Pramod-2021}, that if $\mathcal{F}$ and $\mathcal{G}$ are pure of weight $0$, $\mathcal{F}\boxtimes\mathcal{G}$ is pure of weight $0.$ And $\mathcal{H}$ is pure of weight $0,$ if $\mathcal{F}$ and $\mathcal{G}$ are
pure of weight $0$ and $\mathcal{F}\rightarrow \mathcal{H}\rightarrow\mathcal{G}\xrightarrow{+1}.$ 

We have the proof.

Thus $\mathcal{L}_1$ can be written as a sum of summands of
\[
u_0^*\Ind\left(
 IC(O_{x_P},\overline{\bbQ_l})\boxtimes\prod IC(X(0,1),\overline{\bbQ_l})
\boxtimes\prod IC(O_{nh,\mathrm{semi}},\overline{\bbQ_l})\boxtimes IC(O_{x_I},\overline{\bbQ_l})
\right),
\]
and $\mathcal{L}_1$ is pointwise pure on $U_0\cap X(c,m)$ for any $(c,m).$

Moreover, to study $\Ind\left(
 IC(O_{x_P},\overline{\bbQ_l})\boxtimes\prod IC(X(0,1),\overline{\bbQ_l})
\boxtimes\prod IC(O_{nh,\mathrm{semi}},\overline{\bbQ_l})\boxtimes IC(O_{x_I},\overline{\bbQ_l})
\right)|_{Y(c,m)},$ for $(c,m)$ such that $Y(c,m)\cap U_0\not=\emptyset,$ by Proposition \ref{homext} and \ref{open}, it suffices to consider 
\[
\left\{(x,F)\ \middle|\
\begin{aligned}
&x\in Y(c,m),\\
&F=(0=V_0\subset V_1=V_{2,0}\subset V_{2,1}\subset\cdots
 \subset V_{2,t}=V_{3,0}\\
&\hspace{35mm}\subset V_{3,1}\subset\cdots\subset V_{3,m_1}\subset V),\\
&x(F)\subset F,\qquad M_x|_{V_1/V_0}\cong x_I,\qquad
 \ddim(V_{3,j}/V_{3,j-1})=\delta,\\
&M_x|_{V_{2,i}/V_{2,i-1}}\cong M_{nh,semi},\\
&M_x|_{V/V_{3,r}}\cong x_P
\end{aligned}\right\}\rightarrow Y(c,m) .
\]
By Remark \ref{rem1}, Theorem \ref{local} can apply to this morphism.

We can regard $U_0\cap X(c,m)$ as
\[
U_0\cap\left(
x_P\times x_I\times\prod\bigl(U_{nh,j}\cap X(c_{nh,j},m_j)\bigr)
\times(U_{hr}\cap X(0,\mathbf{m}))\right),
\]
which is open in $
x_P\times x_I\times\prod\bigl(U_{nh,j}\cap X(c_{nh,j},m_j)\bigr)
\times(U_{hr}\cap X(0,\mathbf{m})).$

$U_{hr}\cap X(0,m)$ is smooth.

And regard $U_0\cap Y(c,m)$ as
\[
U_0\cap\left(
x_P\times x_I\times\prod\bigl(U_{nh,j}\cap Y(c_{nh,j},m_j)\bigr)
\times(U_{hr}\cap X(0,\mathbf{m}))\right),
\]
which is open in $
x_P\times x_I\times\prod\bigl(U_{nh,j}\cap Y(c_{nh,j},m_j)\bigr)
\times(U_{hr}\cap X(0,\mathbf{m})).$

$U_{hr}\subset X(0,\mathbf{m})$ is smooth.

Using the smooth morphism $\mathrm{G}_{V_{nh,i}}\times U_{nh,i}\longrightarrow \bfE_{V_{nh,i}},$ we see that, for every $x\in U_{nh,i},$ we have $T_xU_{nh,i}+T_x\mo_x=T_x\bfE_{V_{nh,i}}.$

$U_{nh,j}\cap X(c_{nh,j},m_j)$ and $U_{nh,j}\cap Y(c_{nh,j},m_j)$ are smooth with dimension $\dim U_{nh,j}+\dim X(c_{nh,j},m_j)-\dim \bfE_{V_{nh,j}}$ because the intersections are transverse and Remark \ref{ys}.
Thus $U_0\cap X(c,m)$ and $U_0\cap Y(c,m)$ are smooth with the same dimension.

For $\underline{\lambda}=(\lambda_1,\cdots,\lambda_t),$ consider the proper morphism
\[
W_{c,\underline{\lambda}}=\left\{(x,F)\ \middle|\
\begin{aligned}
&x\in U_0\cap Y(c,m),\\
&F=(0=V_0\subset V_1=V_{2,0}\subset V_{2,1}\subset\cdots
 \subset V_{2,s}=V_{3,0}\\
&\hspace{35mm}\subset V_{3,1}\subset\cdots\subset V_{3,t}\subset V),\\
&x(F)\subset F,\,M_x|_{V_1/V_0}\cong x_I,\qquad
 M_x|_{V_{2,i}/V_{2,i-1}}\cong c_{nh,i},\\
&M_x|_{V_{3,j}/V_{3,j-1}}\text{ is homogeneous regular of dimension }\lambda_j\delta,\\
&M_x|_{V/V_{3,t}}\cong x_P
\end{aligned}\right\}\xrightarrow{t_{c,\underline{\lambda}}}U_0\cap Y(c,m) .
\]

Applying Theorem \ref{local} to $$\Ind\left(
 IC(O_{x_P},\overline{\bbQ_l})\boxtimes\prod IC(X(0,1),\overline{\bbQ_l})
\boxtimes\prod IC(O_{nh,\mathrm{semi}},\overline{\bbQ_l})\boxtimes IC(O_{x_I},\overline{\bbQ_l})
\right)|_{Y(c,m)},$$ we find that $$j_0^{c,m*}\Ind\left(
 IC(O_{x_P},\overline{\bbQ_l})\boxtimes\prod IC(X(0,1),\overline{\bbQ_l})
\boxtimes\prod IC(O_{nh,\mathrm{semi}},\overline{\bbQ_l})\boxtimes IC(O_{x_I},\overline{\bbQ_l})
\right)$$ is generated by the objects ${t_{c,\underline{\lambda}}}_!\overline{\bbq_l}$ in $\cD^b(U_0^{c,m}).$

We claim that $W_{c,\underline{\lambda}}$ is smooth.

To prove the claim, we regard it as an open subset of
\[
\left\{(x,F)\ \middle|\
\begin{aligned}
&x\in x_P\times x_I\times\prod\bigl(U_{nh,j}\cap Y(c_{nh,j},m_j)\bigr)\times U_{hr},\\
&F=(0=V_0\subset V_1=V_{2,0}\subset V_{2,1}\subset\cdots
 \subset V_{2,s}=V_{3,0}\\
&\hspace{35mm}\subset V_{3,1}\subset\cdots\subset V_{3,t}\subset V),\\
&x(F)\subset F,\,M_x|_{V_1/V_0}\cong x_I,\qquad
 M_x|_{V_{2,i}/V_{2,i-1}}\cong c_{nh,i},\\
&M_x|_{V_{3,j}/V_{3,j-1}}\text{ is homogeneous regular of dimension }\lambda_j\delta,\qquad
 M_x|_{V/V_{3,t}}\cong x_P,\\
&\Hom(M_x|_{V_{nh,j}},M_x|_{V_{nh,j'}})=0,\quad j\ne j',\\
&\Hom(M_x|_{V_{nh,j}},M_x|_{\bigoplus V_{hr,i}})=0
\end{aligned}\right\}.
\]
On this variety, for $j\not=j',i\not=i'$, the representations $M_x|_{V_{nh,j}},$ $M_x|_{V_{nh,j'}},$ $M_x|_{V_{hr,i}}$ and $M_x|_{V_{hr,i'}}$ have no direct summands in the same tube.

When they have no direct summands in the same tube, stratifying by $\ddim(V_{nh,j}\cap V_{3,i})$ and $\ddim(V_{hr,j}\cap V_{3,i})$, we obtain a decomposition
in which each stratum is an open and closed subvariety.

This stratum is isomorphisc to some form \[
\begin{aligned}&\left\{(x,F)\ \middle|\
\begin{aligned}
&x\in U_0\cap(x_P\times x_I\times\prod U_{nh,j}\cap Y(c_{nh,j},m_j)\times \prod U_{hr,j}),\\
&F=(0=V_0\subset V_1\subset V^1_{2,0}\subset\cdots
 \subset V^1_{2,m_1’}\subset V^2_{2,0}\subset\cdots
 \subset V^2_{2,m_2’}\subset\cdots\\&\subset V^s_{2,0}\subset\cdots
 \subset V^s_{2,m_s’}= V_{3,0}\subset\cdots \subset V_{3,\mathbf{m}}\subset V)\\& x(F)\subset F, V_1=V_I,\, V_{2,m_i'}^i=V_I\oplus(\bigoplus_{j\leq i}V_{nh,j}),\, V_{3,i}=V_I\oplus(\bigoplus_j V_{nh,j})\oplus(\bigoplus_{s\leq i}V_{hr,s}),\\
&M_x|_{V_1/V_0}\cong x_I,\,M_x|_{V/V_{3,\mathbf{m}}}\cong x_P,\,M_x|_{V^1_{2,0}/V_1}\cong c_{nh,1},\,M_x|_{V^i_{2,0}/V^{i-1}_{2,m_{i-1}'}}\cong c_{nh,i},\,s\geq i>1\\
 &M_x|_{V^j_{2,i}/V^j_{2,i-1}}\text{ is homogeneous regular of dimension } \mu^j_i\delta,\, m_{i}'\geq i\geq 1, x|_{V_{3,j}/V_{3,j-1}}\in U_{hr,j},\, 0<j\leq \mathbf{m}.
\end{aligned}\right\}\\ &\end{aligned}
\]

Now we only need to consider
\[
\begin{aligned}&\prod\left\{(x,F)\ \middle|\
\begin{aligned}
&x\in U_{nh,j}\cap Y(c_{nh,j},m_j),\\
&F=(0=V_1\subset V_{2,0}\subset V_{2,1}\subset\cdots
 \subset V_{2,m_j’}=V_{nh,j}),x(F)\subset F,\,\\
&M_x|_{V_{2,0}}\cong c_{nh,j},\qquad
 M_x|_{V_{2,i}/V_{2,i-1}}\text{ is homogeneous regular of dimension } \mu^j_i\delta
\end{aligned}\right\}\\ &\times \prod U_{hr,j}\end{aligned}
\]
since the stratum we construct is isomorphisc to an open subset in it.

It is clearly that $U_{hr,j}=X(0,1)\cap \cT_{hr,j}$ is smooth.

And $$\left\{(x,F)\ \middle|\
\begin{aligned}
&x\in Y(c_{nh,j},m_j),\\
&F=(0=V_1\subset V_{2,0}\subset V_{2,1}\subset\cdots
 \subset V_{2,m_j'}=V_{nh,j}),x(F)\subset F,\,\\
&M_x|_{V_{2,0}}\cong c_{nh,j},\qquad
 M_x|_{V_{2,i}/V_{2,i-1}}\text{ is homogeneous regular of dimension } \mu^j_i\delta
\end{aligned}\right\}$$ is smooth by the same way in Proposition \ref{prop:density}.

The intersection of smooth varieties
\[
\begin{aligned}
&\left\{(x,F)\ \middle|\
\begin{aligned}
&x\in Y(c_{nh,j},m_j),\\
&F=(0=V_1\subset V_{2,0}\subset V_{2,1}\subset\cdots
 \subset V_{2,m_j'}=V_{nh,j}),x(F)\subset F,\,\\
&M_x|_{V_{2,0}}\cong c_{nh,j},\qquad
 M_x|_{V_{2,i}/V_{2,i-1}}\text{ is homogeneous regular of dimension } \mu^j_i\delta
\end{aligned}\right\}\\
&\qquad\cap U_{nh,j}\times\mathcal F_{c_{nh,j},\mu^j},
\end{aligned}
\]
where $\mathcal F_{c_{nh,j},\mu^j}$ is the smooth variety 
\[
\left\{F=(0=V_{-1}\subset V_0\subset V_1\subset\cdots\subset V_{m_j'}=V_{nh,j})\ \middle|\
\begin{aligned}
&\ddim V_i/V_{i-1}=\mu^j_i\delta,\,0<i\leq m_j'\\
&\ddim V_0/V_{-1}=\ddim c_{nh,j}
\end{aligned}\right\},
\]
 is transverse in $\bfE_{V_{nh,j}}\times\mathcal F_{c_{nh,j},\mu^j}$. The proof is similar to the proof that $U_{nh,j}\cap Y(c_{nh,j},m_j)$ is a transversal intersection. The intersection is therefore smooth, which proves the claim. 

More precisely, we can compute the dimension of
\[
\begin{aligned}&\prod\left\{(x,F)\ \middle|\
\begin{aligned}
&x\in U_{nh,j}\cap Y(c_{nh,j},m_j),\\
&F=(0=V_1\subset V_{2,0}\subset V_{2,1}\subset\cdots
 \subset V_{2,m_j'}=V_{nh,j}),x(F)\subset F,\,\\
&M_x|_{V_{2,0}}\cong c_{nh,j},\qquad
 M_x|_{V_{2,i}/V_{2,i-1}}\text{ is homogeneous regular of dimension } \mu^j_i\delta
\end{aligned}\right\}\\ &\times \prod U_{hr,j}\end{aligned}
\]
which is $\sum_{j}\dim(U_{nh,j}\cap X(c_{nh,j},m_j))+ \dim U_{hr},$ the same as the dimension of $U_0\cap X(c,m).$
Thus $\dim W_{c.\lambda}=\dim U_0\cap X(c,m).$

By Remark \ref{lisse},
\[
\widetilde j_0^{c,m*}\Ind\left(
 IC(O_{x_P},\overline{\bbQ_l})\boxtimes\prod IC(X(0,1),\overline{\bbQ_l})
\boxtimes\prod IC(O_{nh,\mathrm{semi}},\overline{\bbQ_l})\boxtimes IC(O_{x_I},\overline{\bbQ_l})
\right)
\]
is a lisse sheaf complex. Since the variety is smooth, pointwise purity implies purity by \cite[6.2.5]{PD}.

We now only need to consider the connected component $W$ of
\[
\left\{(x,F)\ \middle|\
\begin{aligned}
&x\in U_0\cap Y(c,m),\\
&F=(0=V_0\subset V_1=V_{2,0}\subset V_{2,1}\subset\cdots
 \subset V_{2,s}=V_{3,0}\\
&\hspace{35mm}\subset V_{3,1}\subset\cdots\subset V_{3,t}\subset V),\\
&x(F)\subset F,\,M_x|_{V_1/V_0}\cong c_I,\qquad
 M_x|_{V_{2,i}/V_{2,i-1}}\cong c_{nh,i},\\
&M_x|_{V_{3,j}/V_{3,j-1}}\text{ is homogeneous regular of dimension }\lambda_j\delta,\qquad
 M_x|_{V/V_{3,t}}\cong c_P
\end{aligned}\right\}.
\]
It has the same dimension as $U_0\cap X(c,m)$.

The morphism
$W\xrightarrow{\phi}U_0\cap Y(c,m)$ is small respect to $U_0\cap X(c,m)$ for the following reasons. 

Properness follows from the construction. And for any $x\in U_0\cap Y(c,m),$ $j\not=j',$ $i\not=i',$ $M_x|_{V_{nh,j}},$ $M_x|_{V_{nh,j'}},$ $M_x|_{V_{hr,i}}$ and $M_x|_{V_{hr,i'}}$ have no direct summands in the same tube. Then $U_0\cap X(c,m)= U_0\cap (x_p\times x_I\times \Pi_jX(c_{nh,j},m_j)\times U_{hr}).$
\[
\forall y\in U_0\cap X(c,m),\quad
\phi^{-1}(y)\text{ is a finite set}.
\]
$\exists$ a stratification $X(c_{nh.j},m_j)\sqcup (\bigsqcup Y_{nh,j,t})$ of $Y(c_{nh,j},m_j)$ as in Lemma \ref{small} and \cite[Lemma 2.5]{WX2026}, such that
\[
\dim (X(c_{nh.j},m_j)\cap U_{nh,j})-\dim (Y_{nh,j,t}\cap U_{nh,j})=\dim(X(c_{nh.j},m_j))-\dim(Y_{nh,j,t}),
\]
because $U_{nh,j}$ is transverse to both $Y_{nh,j,t}$ and $X(c_{nh.j},m_j).$ And for any $Y_{nh,j,t}$ other than $X(c_{nh.j},m_j),$  $\dim (Y_{nh,j,t}\cap U_{nh,j})<\dim (X(c_{nh.j},m_j)\cap U_{nh,j}).$ Morever, $$U_0\cap Y(c,m)=(U_0\cap X(c,m))\cup \bigcup_{finite} (U_0\cap (x_p\times x_I\times \Pi_{j}(Y_{nh,j,t_j}\cap U_{nh,j}) \times U_{hr})).$$ By calculating the dimension, $U_0\cap X(c,m)$ is dense in $U_0\cap Y(c,m).$

Applying Lemma \ref{small} and \cite[Lemma 2.5]{WX2026}, for $y\in U_0\cap (x_p\times x_I\times \Pi_{j}(Y_{nh,j,t_j}\cap U_{nh,j}) \times U_{hr}),$ when there exist $j$ such that $Y_{nh,j,t_j}\not=X(c_{nh,j},m_j),$ $\dim \phi^{-1}(y)<\frac{1}{2}\sum_j (\dim X(c_{nh.j},m_j)- \dim Y_{nh,j,t}),$ since $M_x|_{V_{nh,j}}$ and $M_x|_{V_{nh,j‘}}$ have no direct summands in the same tube for different $j$ and $j'.$

Thus the morphism
$W\xrightarrow{\phi}U_0\cap Y(c,m)$ is small respect to $U_0\cap X(c,m).$ 

Let $\widetilde U_0^{c,m}\xrightarrow{j_{c,m}}U_0^{c,m}$ denote the open embedding. And $L_{\chi}|_{U_0\cap X(c,m)}$ is lisse sheaf and pointwise pure. With the smoothness of $U_0\cap X(c,m),$ $L_{\chi}|_{U_0\cap X(c,m)}$ is pure and thus semisimple in $\cD^b(\widetilde{U}_0^{c,m}).$ Morever, we know that $L_{\chi}|_{U_0\cap X(c,m)}$ is semisimple local system in $\cD^b(\widetilde{U}_0^{c,m}).$

${t_{c,\underline{\lambda}}}_!\overline{\bbQ_l}$ is generated by simple perverse sheaf ${j_{c,m}}_{!*}(L'[\dim (U_0\cap X(c,m)]),$ where $L'$ is a simple local system summand of $L_{\chi}|_{U_0\cap X(c,m)},$ for some $\chi\in X(S_m),$ in $\cD^b(U_0^{c,m}).$ 

Now set
\[
K=\Ind\left(
 IC(O_{x_P},\overline{\bbQ_l})\boxtimes\prod IC(X(0,1),\overline{\bbQ_l})
\boxtimes\prod IC(O_{nh,\mathrm{semi}},\overline{\bbQ_l})\boxtimes IC(O_{x_I},\overline{\bbQ_l})
\right).
\]

The purity of $\widetilde j_0^{c,m*}K$ implies that $j_0^{c,m*}K$ is pure, as we now show.

$j_0^{c,m*}K$ is generated by elements ${t_{c,\underline{\lambda}}}_!\overline{\bbQ_l}.$ By the smallness of the morphism $\phi,$ every simple perverse generator of $j_0^{c,m*}K$
has the form $j_{c,m!*}P$, where $P$ is a simple perverse sheaf on $\widetilde U_0^{c,m}$. Moreover, $P$ is a local system up to a shift.

Now consider ${}^p\mathbf{H}^i(j_0^{c,m*}K)$ and its weight filtration:
\[
0\subset P_1\subset P_2\subset\cdots\subset{}^p\mathbf{H}^i(j_0^{c,m*}K).
\]
The functor $j_{c,m}^*$ preserves this weight filtration.
Moreover, ${}^p\mathbf{H}^i(\widetilde j_0^{c,m*}K)\simeq \mathbf{H}^{i-\dim U_0^{c,m}}(\widetilde j_0^{c,m*}K)[\dim  U_0^{c,m}],$ since $\mathbf{H}^{k}(\widetilde j_0^{c,m*}K)$ is a lisse sheaf complex for any $k.$

This complex is pure of weight $i.$

Moreover, $j_{c,m}^*(P_i/P_j)=0$ implies $P_i/P_j=0$, since every generator
has the form $j_{c,m!*}P$ and the faithfulness in \cite[Proposition 5.3.9]{Pramod-2021}.

It follows that ${}^p\mathbf{H}^i(j_0^{c,m*}K)$ is pure of the
same weight as ${}^p\mathbf{H}^i(\widetilde j_0^{c,m*}K)$. Thus $j_0^{c,m*}K$
is also pure and hence semisimple in $\cD^b(U_0^{c,m})$. Consequently, $j_0^{c,m*}\mathcal{L}_1$ is semisimple in $\cD^b(U_0^{c,m})$.

It follows that $j_0^{c_2,m_2*}\mathcal{L}$ is semisimple and by Lemma \ref{simp}, $\mathcal{L}_2$ can be written as a sum of summands(up to shift) of
\[
u_0^*\Ind\left(
 IC(O_{x_P},\overline{\bbQ_l})\boxtimes\prod IC(X(0,1),\overline{\bbQ_l})
\boxtimes\prod IC(O_{c^{2}_{nh,j}},\overline{\bbQ_l})\boxtimes IC(O_{x_I},\overline{\bbQ_l})
\right),
\]
By descending induction on the support of $\mathcal{L}_i$, we conclude that $\widetilde j_0^{c,m*}\operatorname{For}^{\bfG_{\nu}}\mathcal{L}$
is pointwise pure. Thus $\mathcal{L}$ is pointwise pure on $X(\mathbf{c},\mathbf{m}).$
\end{proof}
\subsection*{Acknowledgements} 
Y. Wu and J. Xiao are partially supported by National Natural Science Foundation of China [Grant No. 12471030]. An initial version of the proof of Lemma~\ref{lem:grassmannian-pushforward} was developed with AI assistance, and its correctness was subsequently verified by the authors. All other mathematical content was developed independently of AI. AI tools were also used to correct grammar and spelling.
\bibliography{ref}
\end{document}